\documentclass[11pt, twoside, leqno]{article}

\usepackage{amssymb}
\usepackage{amsmath}
\usepackage{amsthm}
\usepackage{color}
\usepackage{mathrsfs}
\usepackage{indentfirst}
\usepackage{txfonts}

\allowdisplaybreaks

\def\red{\color{red}}

\def\rr{{\mathbb R}}
\def\rn{{\mathbb{R}^n}}

\def\zz{{\mathbb Z}}
\def\cc{{\mathbb C}}

\def\nn{{\mathbb N}}

\def\CB{{\mathcal B}}

\def\cm{{\mathcal M}}

\def\cp{{\mathcal P}}

\def\cs{{\mathcal S}}

\def\fz{\infty }
\def\az{\alpha}

\def\lz{\lambda}

\def\tz{\theta}
\def\sz{\sigma}

\def\lf{\left}
\def\r{\right}

\def\la{\langle}
\def\ra{\rangle}

\def\ls{\lesssim}
\def\gs{\gtrsim}

\def\noz{\nonumber}

\def\st{\subset}

\def\loc{{\mathop\mathrm{\,loc\,}}}
\def\supp{\mathop\mathrm{\,supp\,}}

\def\XXint#1#2#3{{\setbox0=\hbox{$#1{#2#3}{\int}$ }
\vcenter{\hbox{$#2#3$ }}\kern-.6\wd0}}

\def\unp{\underline{p}}

\def\la{{\langle}}
\def\ra{{\rangle}}
\def\({\left(}
\def \){ \right)}

\def\CB{{\mathcal B}}

\def\one{\mathbf{1}}

\def\Hatom{{H_{X,{\rm atom}}^{A,q,d}}}

\newtheorem{theorem}{Theorem}[section]
\newtheorem{lemma}[theorem]{Lemma}
\newtheorem{corollary}[theorem]{Corollary}
\newtheorem{assumption}[theorem]{Assumption}

\theoremstyle{definition}
\newtheorem{remark}[theorem]{Remark}
\newtheorem{definition}[theorem]{Definition}
\renewcommand{\appendix}{\par
\setcounter{section}{0}%
\setcounter{subsection}{0}%
\setcounter{subsubsection}{0}%
\gdef\thesection{\@Alph\c@section}%
\gdef\thesubsection{\@Alph\c@section.\@arabic\c@subsection}%
\gdef\theHsection{\@Alph\c@section.}%
\gdef\theHsubsection{\@Alph\c@section.\@arabic\c@subsection}%
\csname appendixmore\endcsname
}

\numberwithin{equation}{section}
\begin{document}
\title{\bf\Large Fourier Transform of Anisotropic Hardy Spaces Associated
with Ball Quasi-Banach Function Spaces and Its
Applications to Hardy--Littlewood Inequalities
\footnotetext{\hspace{-0.35cm} 2020 {\it Mathematics Subject
Classification}.
Primary 42B10; Secondary 42B30, 42B35, 46E30.
\endgraf {\it Key words and phrases.} expansive matrix,
ball quasi-Banach function space,
anisotropic Hardy space,
Fourier transform, Hardy--Littlewood inequality.
\endgraf This project is partially supported by the National Key Research
and Development Program of China (Grant No. 2020YFA0712900),
the National Natural Science Foundation of China
(Grant Nos. 11971058 and 12071197),
China Postdoctoral Science Foundation (Grant No. 2022M721024),
and the Open Project Program of Key Laboratory of Mathematics
and Complex System of Beijing Normal University (Grant No. K202304).}}
\date{}
\author{Chaoan Li, Xianjie Yan and Dachun Yang\,\footnote{Corresponding
author, E-mail: \texttt{dcyang@bnu.edu.cn}/{\red{July 13, 2023}}/Final version.}}
\maketitle
\date{}
\maketitle

\vspace{-0.8cm}

\begin{center}
\begin{minipage}{13cm}
{\small {\bf Abstract}\quad
Let $A$ be a general expansive matrix and
$X$ be a ball quasi-Banach function space on $\mathbb R^n$,
whose certain power (namely its convexification)
supports a Fefferman--Stein vector-valued
maximal inequality and the associate space of
whose other power supports the boundedness of the powered
Hardy--Littlewood maximal operator.
Let $H_X^A(\mathbb{R}^n)$ be the anisotropic Hardy
space associated with $A$ and $X$.
The authors first prove that the Fourier transform
of $f\in H^A_{X}(\mathbb{R}^n)$ coincides with a
continuous function $F$ on $\mathbb{R}^n$ in the sense of tempered
distributions.
Moreover, the authors obtain a pointwise inequality that the
function $F$ is less than the product of the anisotropic Hardy space norm of $f$ and
a step function with respect to the transpose
matrix of the expansive matrix $A$.
Applying this, the authors further induce a higher order
convergence for the function
$F$ at the origin and give a variant of the Hardy--Littlewood
inequality in
$H^A_{X}(\mathbb{R}^n)$. All these results have a wide range
of applications. Particularly, the authors apply these results,
respectively, to classical
(variable and mixed-norm) Lebesgue spaces,
Morrey spaces, Lorentz spaces, Orlicz spaces,
Orlicz-slice spaces, and  local generalized Herz spaces and,
even on the last five function spaces,
the obtained results are completely new.
}
\end{minipage}
\end{center}


\vspace{0.2cm}
\section{Introduction\label{s0}}
In 1972, Fefferman and Stein \cite{fs72} introduced a famous problem,
that is,
what is the characterization of the Fourier
transform
$\widehat{f}$ of a distribution $f$ from the classical Hardy space $H^p(\rn)$.
Recall that, in 1974, Coifman \cite{coi74} characterized
$\widehat{f}$ via the entire
function
of exponential type for $n=1$, where $f\in H^p(\rr)$ with $p\in(0,1]$.
Since then, many researchers investigated the characterization of
$\widehat{f}$
with the distribution $f$ from Hardy spaces with $n\geq2$;
see,
for instance, \cite{bw13,col82,gk01,tw80}.
In particular, Taibleson and Weiss \cite{tw80}
proved that, for any given $p\in(0,1]$, the Fourier transform of
$f\in H^p(\rn)$
coincides with a continuous function $F$ in the sense of tempered
distributions
and there exists a positive constant $C$, independent of $f$ and $F$,
such that,
for any $x\in\rn$,
\begin{equation}\label{eqabs1}
|F(x)|\leq C\|f\|_{H^p(\rn)}|x|^{n(\frac{1}{p}-1)}.
\end{equation}
This further implies the following generalization of the
Hardy--Littlewood
inequality that
\begin{equation}\label{eqabs2}
\left[\int_\rn|x|^{n(p-2)}|F(x)|^p\,dx\right]^{1/p}
\leq C\|f\|_{H^p(\rn)},
\end{equation}
where $C$ is a positive constant independent of $f$ and $F$
(see \cite[p.\,128]{ste93}).

Recently, Sawano et al. \cite{shyy17} originally introduced
the ball quasi-Banach function space $X$ and the associated Hardy space
$H_X(\rn)$. In their article \cite{shyy17}, by assuming that the
Hardy--Littlewood maximal operator satisfies a
Fefferman--Stein vector-valued inequality on certain power
(namely its convexification) of
$X$ and the powered Hardy--Littlewood maximal operator is
bounded on the associate space
of certain power of $X$,
Sawano et al. established various maximal function characterizations
and several other characterizations of $H_X(\rn)$, respectively,
in terms of atoms, molecules, and Lusin area functions.
Indeed,
the real-variable theory of the Hardy space $H_X(\rn)$
associated with the ball quasi-Banach function space $X$
provides an unified framework  of various
types of Hardy spaces, which includes many important Hardy
spaces that have
been studied before, such as classical
Hardy spaces, mixed-norm Hardy spaces, variable Hardy spaces, and
Orlicz--Hardy spaces. For more recent developments on this topic,
we refer the reader to \cite{cwyz20, hks22,shyy17, wyy,yyy20,zyy21}.
Based on the recent rapid developments of
the theory of the Hardy space $H_X(\rn)$ and the aforementioned works
on the characterization of the Fourier transform of the classical Hardy
space $H^p(\rn)$ or their generalization,
very recently, Huang et al. \cite{hcy22} showed that both
\eqref{eqabs1} and \eqref{eqabs2} hold true in $H_X(\rn)$.

On the other hand, in 2003, motivated by the application of discrete
groups of dilations in wavelet theory,
Bownik \cite{Bownik} introduced and investigated the anisotropic
Hardy space $H^p_{A}(\rn)$ with $p\in(0,\fz)$,
where $A$ is a general expansive matrix on $\rn$,
which includes both the classical
Hardy space and the parabolic Hardy space of Calder\'on and Torchinsky
\cite{ct75} as special cases.
Since then, various variants of classical Hardy spaces
over anisotropic Euclidean spaces (see, for instance,
\cite{blyz08weight,blyz10weight,cgn17,cgn19,cgn19no2,
hcy21,lyy23,lhy12,liu2020,lyz2023}) or, more generally,
over spaces of homogeneous type in the sense of
Coifman and Weiss \cite{CWhomo,CWhomo2} (see, for instance,
\cite{cgp22,gkp21,gn18,hhllyy19,hlyy19})
have been introduced and
their real-variable theories have been well developed.
In 2013, Bownik and Wang \cite{bw13} generalized inequalities \eqref{eqabs1}
and \eqref{eqabs2} to the anisotropic
Hardy space $H_A^p(\rn)$ with the known characterization
of $H_A^p(\rn)$. Furthermore,
Liu \cite{liu2020} pointed out that \eqref{eqabs1}
and \eqref{eqabs2} also apply to the setting of
the variable anisotropic Hardy space
$H_A^{p(\cdot)}(\rn)$.
Later, Liu et al. \cite{llz22} proved that \eqref{eqabs1}
and \eqref{eqabs2} hold true for the
anisotropic mixed-norm Hardy space
$H_A^{\vec{p}}(\rn)$.
Recall that
the anisotropic Hardy space $H_X^A(\rn)$ associated with both $A$
and $X$ was first introduced and studied by Wang et al. \cite{wyy22},
in which Wang et al. characterized $H_X^A(\rn)$
in terms of maximal functions,
atoms, finite atoms, and molecules.
Moreover, the variable anisotropic Hardy space
$H_A^{p(\cdot)}(\rn)$ and the anisotropic mixed-norm Hardy space
$H_A^{\vec{p}}(\rn)$ are the special cases of $H_X^A(\rn)$ with
$X:= L^{p(\cdot)}(\rn)$ or $X: = L^{\vec{p}}(\rn)$, respectively.
Based on these results,
it is natural to ask whether
\eqref{eqabs1} and \eqref{eqabs2}
also hold true for $H_X^A(\rn)$.
The goal of this article is to give a positive answer to this question.

Let $A$ be a dilation and $X$ a ball quasi-Banach function space on $\rn$.
In this article, under the assumptions that
the Hardy--Littlewood maximal
operator satisfies some Fefferman--Stein vector-valued inequality
on certain power of
$X$, the powered Hardy--Littlewood maximal operator is bounded
on the associate
space of certain power of $X$,
and the $X$-quasi-norm of
characteristic functions of anisotropic balls has a lower bound,
we get rid of the dependence on the concavity of $\|\cdot\|_X$.
With these mild assumptions
and two uniform pointwise estimates
we show that the Fourier
transform $\widehat{f}$ of $f\in H_X^A(\rn)$
coincides with a continuous function $F$ on $\rn$
in the sense of tempered distributions and
prove that an
inequality similar to \eqref{eqabs1} also holds true
for any $f\in H_X^A(\rn)$.
Furthermore,
applying this and a technical inequality about the
value of the Fourier transform of atoms,
we further conclude a higher order convergence
of the continuous function $F$ at the origin
and then show that an inequality
similar to \eqref{eqabs2} holds true for
$H_X^A(\rn)$,
which is a variant of the Hardy--Littlewood inequality in
$H_X^A(\rn)$.
It is remarkable that the results obtained in this article
have a wide range of generality because ball quasi-Banach
function spaces include lots of important
function spaces.
In particular, when these results are applied
to classical Lebesgue spaces,
variable Lebesgue spaces, and mixed-norm Lebesgue spaces,
the obtained conclusions coincide with the known ones;
morever, when these results are applied, respectively,
to Morrey spaces, Lorentz spaces,
Orlicz spaces,
Orlicz-slice spaces,
and  local generalized Herz spaces,
the obtained conclusions are completely new.
More applications of these results to new-found function
spaces are quite possible.

The remainder of this article is organized as follows.

In Sect. \ref{s1}, we first present definitions of expansive matrices,
ball quasi-Banach function spaces, and anisotropic Hardy spaces
associated with both $A$ and
$X$; see Definitions \ref{dilation}, \ref{BQBFS}, and \ref{HXA} below.

The aim of Sect. \ref{s2} is to prove the main result
(see Theorem \ref{s2t1} below), that is,
the Fourier transform $\widehat{f}$ of $f\in H_X^A(\rn)$
coincides with a continuous function $F$ in the sense of tempered
distributions. In order to achieve this,
we apply Lemmas \ref{s2l1} (some subtle estimates on derivatives of the
Fourier transform of the dilation of atoms)
and \ref{s2l3}
(some exquisite relations between the Euclidean norm and the step
homogeneous quasi-norm $\rho$ under consideration) to
establish a uniform pointwise estimate for atoms
(see Lemma \ref{s2l2} below).
Then Theorem \ref{s2t1} is proved by
this and some real-variable characterizations from \cite{wyy22},
especially its atomic decompositions.
Morever, we apply these results, respectively,
to classical Lebesgue spaces,
variable Lebesgue spaces, and mixed-norm Lebesgue spaces
and show that the obtained conclusions coincide with the known ones
[see Remark \ref{s2re1}(i)-(iv) below].
At the same time, we obtain a pointwise inequality
of the continuous function $F$, which suggests that the
anisotropic mixed-norm atoms must possess
vanishing moments in some
sense [see Remark \ref{s2re1}(v) below].

Applying the Fourier transform, in Sect. \ref{s3},
we present some further applications of Theorem \ref{s2t1}.
First, we prove that the above function $F$ has a higher order
convergence at
the origin
(see Theorem \ref{s3t1} below).
Second, we show that the term
$$|F(\cdot)|\min\left\{\left[\rho_{*}(\cdot)\right]^{1-\frac 1{p_-}-
\frac 1{q_0}+(d+1)\frac{\ln\lambda_-}{\ln b}},\,
\left[\rho_{*}(\cdot)\right]^{1-\frac 2{q_0}+(d+1)
\frac{\ln\lambda_-}{\ln b}}\right\}$$
is $L^{q_0}(\rn)$-integrable and a positive constant
multiple of the anisotropic Hardy space norm of $f$
can uniformly
controll this integral. Thus, we extend the
Hardy--Littlewood inequality to the setting of
anisotropic Hardy spaces
associated with ball quasi-Banach function spaces
(see Theorem \ref{s3t2} below).
Further, we apply these results,
respectively, to classical Lebesgue spaces,
variable Lebesgue spaces, and mixed-norm Lebesgue spaces
and prove that the obtained conclusions coincide with the known ones
(see Remark \ref{6.11rem} below).

As applications, in Sect. \ref{s4},
we apply Theorems \ref{s2t1}, \ref{s3t1}, and \ref{s3t2},
via verfiying all the necessary assumptions,
to five concrete examples of ball quasi-Banach function spaces,
namely
Morrey spaces (see Subsection \ref{s6-appl1} below),
Lorentz spaces (see Subsection \ref{s6-appl3} below),
Orlicz spaces (see Subsection \ref{s6-appl7} below),
Orlicz-slice spaces (see Subsection \ref{s6-appl2} below), and
local Herz--Hardy spaces (see Subsection \ref{s6-appl6} below).
In particular, we show that the anisotropic local Herz space
is a quasi-banach function space (see Theorem \ref{5.27.x1} below).
Through both a boundedness criterion of sublinear
operators on anisotropic local generalized Herz spaces
and its simple corollary (see Lemma \ref{vmbhl-lem}
and Corollary \ref{vmbhl-cor} below), we prove that the
anisotropic local generalized Herz space
supports a Fefferman--Stein vector-valued
maximal inequality and the associate space of
whose other power supports the boundedness of the powered
Hardy--Littlewood maximal operator (see Lemma \ref{vmbhl}
and Theorem \ref{mbhal} below).

At the end of this section, we make some conventions on notation.
Let $\nn:=\{1,2,\ldots\}$, $\zz_+:=\nn\cup\{0\}$, $\zz_+^n:=(\zz_+)^n$,
and $\bf{0}$ be the \emph{origin} of $\rn$.
For any multi-index $\az:=(\az_1,\ldots,\az_n)\in\mathbb{Z}_+^n$
and any $x:=(x_1,\ldots,x_n)\in\rn$, let $|\az|:=\az_1+\cdots+\az_n$,
$\partial^{\az}:=(\frac{\partial}{\partial x_1})^{\az_1}\cdots
(\frac{\partial}{\partial x_n})^{\az_n},$ and
$x^\az:=x_1^{\az_1}\cdots x_n^{\az_n}.$
We denote by $C$ a \emph{positive constant}
which is independent of the main parameters involved, but may vary from
line to line. We use $C_{(\az,\dots)}$ to denote a positive
constant depending on the indicated parameters $\az,\, \dots$.
The symbol $f\ls g$ means $f\leq  Cg$. If $f\ls g$ and $g\ls f$,
we then write $f\sim g$.
If $f \le C g$ and $g=h$ or $g\le h$, we then write $f\ls g = h$ or $f\ls g\le h$.
For any $q\in [1,\infty]$, we denote by $q'$ its \emph{conjugate
index}, that is,
$1/q+1/q'=1$. For any $x\in\rn$,
we denote by $|x|$ the $n$-dimensional
\emph{Euclidean metric} of $x$.
If $E$ is a subset of $\rn$,
we denote by ${\mathbf{1}}_E$ its \emph{characteristic function}
and by $E^\complement$ the set $\rn\setminus E$.
For any $r\in(0,\fz)$ and $x\in\rn$, we denote by $B(x,r)$
the ball centered at $x$ with the radius $r$, that is,
$B(x,r):=\{y\in\rn:\ |x-y|<r\}.$
For any ball $B$, we use $x_B$ to denote its center and $r_B$
its radius and we denote by $\lz B$ for any $\lz\in(0,\fz)$ the
ball concentric with $B$ having the radius $\lz r_B$.
We also use $\epsilon\to 0^{+}$
to denote $\epsilon\in(0,\infty)$ and
$\epsilon \to 0$. Let $X$ and $Y$ be two
normed vector spaces, respectively, with the norm
$\|\cdot\|_X$ and the norm $\|\cdot\|_Y$;
then we use $X\hookrightarrow Y$ to denote
$X\subset Y$ and there exists a positive constant
$C$ such that, for any $f\in X$,
$\|f\|_Y \le C \|f\|_X.$
At last, when we prove a theorem or the like,
we always use the same symbols in the wanted proved
theorem or the like.

\section{Preliminaries\label{s1}}

In this section, we first recall some symbols and concepts
on dilations (see, for instance, \cite{Bownik,shyy17}) as well
as ball quasi-Banach function spaces (see, for instance,
{\cite{shyy17,wyy,wyyz,yyy20,zwyy}}). We begin by recalling
the concept of expansive matrices from \cite{Bownik}.

\begin{definition}\label{dilation}
A real $ n\times n $ matrix $ A $ is called an {\it expansive matrix}
(shortly, a {\it dilation}) if
\begin{equation*}
\min_{\lz \in \sz(A)} |\lz|>1,
\end{equation*}
here and hereafter, $ \sz (A) $ denotes the {\it set of all
eigenvalues of $ A $}.
\end{definition}

Let $A$ be a dilation and
\begin{align}\label{2.14.x1}
b:=|\det A|,
\end{align}
where $\det A$ denotes the determinant of $A$.
Then it follows from \cite[p.\ 6, (2.7)]{Bownik}
that $b\in(1,\infty)$. By the fact that there exists an open and symmetry
ellipsoid $\Delta$, with $ |\Delta|=1 $, and an $ r\in(1,\fz) $ such that
$ \Delta \subset r\Delta \subset A\Delta $ (see \cite[p.\,5, Lemma 2.2]{Bownik}),
we find that, for any $ k\in \zz $,
\begin{equation}\label{B_k}
B_k := A^k \Delta
\end{equation}
is open, $ B_k \subset r B_k \subset B_{k+1} $, and $ |B_k|=b^k $.
For any $ x \in\rn $ and $ k \in \zz $, an ellipsoid $ x+B_k $ is
called a {\it dilated ball}. In what follows, we always let $ \CB $
be the set of all such dilated balls, that is,
\begin{equation}\label{ball-B}
\CB:=\{x+B_k :\  x \in \rn,k \in \zz \}
\end{equation}
and let
\begin{equation}\label{tau}
\tau := \inf\left\{ l \in\zz :\  r^l \geq 2\right\}.
\end{equation}

Let $\lz_-,\lz_+ \in (0,\infty)$ satisfy that
\begin{equation*}
1<\lz_- <\min\{ |\lz|:\ \lz \in\sigma(A)\}
\leq\max\{|\lz|:\ \lz\in\sigma(A)\}<\lz_+.
\end{equation*}
We point out that, if $A$ is diagonalizable over $\rr$,
then we may let
$$\lz_-:=\min\{|\lz|:\ \lz\in\sz(A)\}\ \mbox{and}\
\lz_+:=\max\{|\lz|:\ \lz\in\sigma(A)\}.$$
Otherwise, we may choose them sufficiently close to
these equalities in accordance with what we need in our arguments.

The following definition of the homogeneous quasi-norm is just
\cite[p.\,6,\ Definition 2.3]{Bownik}.

\begin{definition}\label{quasi-norm}
A {\it homogeneous quasi-norm}, associated with a dilation $A$,
is a measurable mapping $\varrho:\ \rn\rightarrow[0,\infty)$ such that
\begin{enumerate}
\item[{\rm(i)}] $\varrho (x)=0\Longleftrightarrow x=\bf{0}$,
where $\bf{0}$ denotes the origin of $\rn$;

\item[{\rm(ii)}] $\varrho(Ax)=b\varrho(x)$ for any $x\in\rn$;

\item[{\rm(iii)}] there exists an $A_0\in[1,\infty)$ such that,
for any $x,y\in\rn$,
$$\varrho(x+y)\leq A_0\,[\varrho(x)+\varrho(y)].$$
\end{enumerate}
\end{definition}

In the standard Euclidean space case,
let $A:=2\,I_{n\times n}$ and, for any $x\in\rn$, $\varrho(x):=|x|^n$.
Then $\varrho$ is an example of homogeneous quasi-norms
associated with $A$ on $\rn$. Here and thereafter,
$I_{n\times n}$ always denotes the $n\times n$ \emph{unit
matrix} and $|\cdot|$ the \emph{Euclidean norm} in $\rn$.

For a fixed dilation $A$, by \cite[p.\,6,\ Lemma 2.4]{Bownik},
we introduce the following quasi-norm
which is used throughout this article.

\begin{definition}\label{def-shqn}
Define the {\it step homogeneous quasi-norm $\rho$} on $\rn$,
associated with the dilation $A$, by setting
\begin{equation*}
\rho(x):= \left\lbrace
\begin{aligned}
&b^k\  &&{\rm if}\  x \in B_{k+1}\setminus B_k,\\
&0\  &&{\rm if}\  x=\bf{0},
\end{aligned}
\right.
\end{equation*}
where $b$ is the same as in \eqref{2.14.x1} and,
for any $k\in\zz$, $B_k$ the same as in \eqref{B_k}.
\end{definition}

Then $(\rn,\rho,dx)$ is a space of homogeneous type
in the sense of Coifman and Weiss \cite{CWhomo},
where $dx$ denotes the $n$-dimensional Lebesgue measure.
For more studies on the real-variable theory of function
spaces over spaces of homogeneous type,
we refer the reader to \cite{bdn20,
bdl18,bdl20,bl11,lj10,lj11,lj13,yhyy2,yhyy,
yy23}.

Throughout this article, we always let $A$ be a dilation
in Definition \ref{dilation}, $b$ the same as in \eqref{2.14.x1},
$\rho$ the step homogeneous quasi-norm in Definition \ref{def-shqn},
$\CB$ the set of all dilated balls in \eqref{ball-B},
$\mathscr M(\rn)$ the {\it set of all measurable functions} on $\rn$
and, for any $k\in\zz$, $B_k$ the same as in \eqref{B_k}.
Now, we recall the definition of ball quasi-norm
Banach function spaces (see \cite{shyy17}).

\begin{definition}\label{BQBFS}
A quasi-normed linear space $X\subset\mathscr M(\rn)$,
equipped with a quasi-norm $\|\cdot\|$ which
makes sense for the whole $\mathscr M(\rn)$, is called
a \emph{ball quasi-Banach function space} if it satisfies
\begin{enumerate}
\item[{\rm(i)}] for any $f\in\mathscr M(\rn)$, $\|f\|_X=0$
implies that $f=0$ almost everywhere;

\item[{\rm(ii)}] for any $f,g\in\mathscr M(\rn)$, $|g|\le|f|$ almost everywhere
implies that $\|g\|_X\le\|f\|_X$;

\item[{\rm(iii)}] for any $\{f_m\}_{m\in\nn}\subset\mathscr M(\rn)$
and $f\in\mathscr M(\rn)$, $0\le f_m\uparrow f$ as $m\to\infty$
almost everywhere implies that
$\|f_m\|_X\uparrow\|f\|_X$ as $m\to\infty$;

\item[{\rm(iv)}] $\one_B \in X$ for any dilated ball $B \in \CB$.
\end{enumerate}

Moreover, a {ball quasi-Banach function} space $X$ is
called a \emph{ball Banach function space} if it satisfies:
\begin{enumerate}
\item[{\rm(v)}] for any $f,g\in X$, $\|f+g\|_X\leq\|f\|_X+\|g\|_X$;

\item[{\rm(vi)}] for any given dilated ball $B\in\CB$,
there exists a positive constant $C_{(B)}$ such that,
for any $f\in X$,
\begin{equation*}
\int_B |f(x)|\,dx \leq C_{(B)}\|f\|_X.
\end{equation*}
\end{enumerate}
\end{definition}

Now, we recall the concept of the $p$-convexification
of ball quasi-Banach function spaces,
which is a part of \cite[Definition 2.6]{shyy17}.

\begin{definition}\label{Debf}
Let $X$ be a ball quasi-Banach function space and $p\in(0,\infty)$.
The \emph{$p$-convexification} $X^p$ of $X$ is defined by setting
$$X^p:=\lf\{f\in\mathscr M(\rn):\ |f|^p\in X\r\}$$
equipped with the \emph{quasi-norm} $\|f\|_{X^p}:=\||f|^p\|_X^{1/p}$
for any $f\in X^p$.
\end{definition}

The associate space $X'$ of any given ball
Banach function space $X$ is defined as follows;
see \cite[Chapter 1, Section 2]{bs88} or \cite[p.\,9]{shyy17}.

\begin{definition}
For any given ball Banach function space $X$,
its \emph{associate space} (also called the
\emph{K\"othe dual space}) $X'$ is defined by setting
\begin{equation*}
X':=\lf\{f\in\mathscr M(\rn):\ \|f\|_{X'}
:=\sup_{g\in X,\ \|g\|_X=1}\lf\|fg\r\|_{L^1(\rn)}<\infty\r\},
\end{equation*}
where $\|\cdot\|_{X'}$ is called the \emph{associate norm} of $\|\cdot\|_X$.
\end{definition}

Now, we recall the concept of the Hardy--Littlewood maximal operator.
In what follows, for any given
$p\in (0,\fz]$ and any given subset $E\st\rn$, let
$L_{\rm loc}^p(E)$ denote the
{\it set of all $p$-order locally integrable functions} on $E$.
Recall that the {\it Hardy--Littlewood maximal operator}
$\cm(f)$ of $f \in L_{\rm loc}^1(\rn)$ is defined by setting,
for any $x \in \rn$,
\begin{align}\label{5.23.x3}
\cm(f)(x):=\sup_{k\in\zz}\sup_{y\in x+B_k}
\frac{1}{|B_k|}\int_{y+B_k} |f(z)|\,dz
=\sup_{x\in B\in\CB}\frac{1}{|B|}\int_B|f(z)|\,dz,
\end{align}
where $\CB$ is the same as in {\eqref{ball-B}} and the supremum
in the second equality is taken over all the balls $B\in\CB$.
For any given $\alpha \in (0,\infty)$,
the {\it powered Hardy--Littlewood maximal operator}
$\cm^{(\alpha)}$ is defined by setting,
for any $f \in L^1_{\rm loc}(\rn)$ and $x \in \rn$,
\begin{equation*}
\cm^{(\alpha)}(f)(x)
:=\lf\{\cm\lf(|f|^\alpha\r)(x)\r\}^{\frac{1}{\alpha}}.
\end{equation*}

Throughout this article, we also need the following fundamental assumptions about
the boundedness of $\cm$ on the convexification of the
given ball quasi-Banach
function space and the boundedness of certain powered of $\cm$
on the associate space of its convexification.

\begin{assumption}\label{F-SVMI}
Let $X$ be a ball quasi-Banach function space.
Assume that there exists a $p_- \in (0,\infty)$ such that,
for any $p \in (0,p_-)$ and $u \in (1,\infty)$,
there exists a positive constant $C$, depending on both $p$ and $r$,
such that, for any $\{f_k\}_{k=1}^\infty\subset\mathscr M(\rn)$,
\begin{equation*}
\left\| \left\{\sum_{k=1}^\infty\lf[\cm\lf(f_k\r)\r]^u\right\}^
{\frac{1}{u}} \right\|_{X^{\frac{1}{p}}}
\leq C\left\|\left\{\sum_{k=1}^{\infty}|f_k|^u\right\}^
{\frac{1}{u}}\right\|_{X^\frac{1}{p}}.
\end{equation*}
\end{assumption}

In what follows, for any given $p_- \in (0,\infty)$, let
\begin{equation} \label{underp}
\underline{p}:=\min \{p_-,1\}.
\end{equation}

\begin{assumption}\label{HMLonAsso}
Let $p_- \in (0,\infty)$ and $X$ be a
ball quasi-Banach function space.
Assume that there exists a $\theta_0\in(0,\underline{p})$,
with $\underline{p}$ in \eqref{underp},
and a $p_0 \in (\theta_0,\infty)$ such that
$X^{{1}/{\theta_0}}$ is a ball Banach function space and,
for any $f\in(X^{{1}/{\theta_0}})'$,
\begin{equation*}
\left\|\cm^{(({p_0}/{\theta_0})')}(f)\right\|_{(X^{{1}/{\theta_0}})'}
\leq C\|f\|_{(X^{{1}/{\theta_0}})'},
\end{equation*}
where $C$ is a positive constant, independent of $f$, and
$\frac{1}{p_0/\theta_0}+\frac{1}{(p_0/\theta_0)'}=1$.
\end{assumption}
Next, recall that a {\it Schwartz function} is a function
$\varphi \in C^\infty(\rn)$ satisfying that, for any
$k\in\zz_+$ and any multi-index $\alpha\in\zz_+^n$,
\begin{align}\label{6.11.x1}
\|\varphi\|_{\alpha,k}
&:=\sup_{x\in\rn}[\rho(x)]^k|\partial^\alpha\varphi(x)|<\infty.
\end{align}
Denote by $ \cs(\rn) $ the {\it set of all Schwartz functions},
equipped with the topology determined by
$\{\|\cdot\|_{\alpha,k} \}_{\alpha\in\zz_+^n,k \in \zz_+}$.
Then $\cs'(\rn)$ is defined to be the {\it dual space} of $\cs(\rn)$,
equipped with the weak-$\ast$ topology. For any $N \in \zz_+$,
let
\begin{equation*}
\cs_N(\rn):=\lf\{\varphi\in\cs(\rn):\
\|\varphi\|_{\alpha,k}\leq1,|\alpha|\leq N,k\leq N\r\},
\end{equation*}
equivalently,
\begin{align*}
&\varphi\in\cs_N(\rn)\ \Longleftrightarrow\ \|\varphi\|_{\cs_N(\rn)}
:=\sup_{|\alpha|\leq N}\sup_{x\in\rn}\max\{1,[\rho(x)]^N\}
|\partial^\alpha\varphi(x)|\leq 1.
\end{align*}

Now, we recall the definitions
of the anisotropic non-tangential maximal function and the
anisotropic
non-tangential grand
maximal function from \cite[Definition 2.15]{wyy22}.
In what follows, for any $\varphi\in\cs(\rn)$ and
$k \in \zz$, let $\varphi_k(\cdot):=b^{-k}\varphi(A^{-k}\cdot)$.
\begin{definition}
Let $\varphi\in\cs(\rn)$ and $f\in\cs'(\rn)$.
The {\it anisotropic non-tangential maximal function} $M_\varphi (f)$,
with respect to $\varphi$, is defined by setting,
for any $x\in\rn$,
\begin{equation*}
M_\varphi(f)(x):=\sup_{k\in\zz,\,y\in x+B_k}\lf|f\ast\varphi_k(y)\r|.
\end{equation*}
Moreover, for any given $ N\in\nn$, the {\it anisotropic
non-tangential grand
maximal function} $M_N(f)$ is defined by setting, for any $x\in\rn$,
\begin{equation}\label{M_N}
M_N(f)(x):=\sup_{\varphi\in\cs_N(\rn)}M_\varphi(f)(x).
\end{equation}
\end{definition}

We present the definition of
$H_{X}^{A}(\rn)$ from \cite{wyy22} as follows. In what follows,
for any $\az\in\rr$, we denote by the \emph{symbol}
$\lfloor\az\rfloor$ the largest integer not greater than $\az$.

\begin{definition}\label{HXA}
Let $A$ be a dilation and $X$ a ball quasi-Banach function space
satisfying both Assumption \ref{F-SVMI} with $p_-\in(0,\infty)$
and Assumption \ref{HMLonAsso} with the same
$p_-$, $\theta_0\in(0,\unp)$, and $p_0\in(\theta_0,\infty)$,
where $\unp$ is the same as in \eqref{underp}. Assume that
$$N\in\nn\cap\lf[\lf\lfloor\lf(\frac{1}{\theta_0}-1\r)
\frac{\ln b}{\ln(\lz_-)}\r\rfloor+2,\infty\r).$$
The {\it anisotropic Hardy space $H_{X,\,N}^A(\rn)$},
associated with both $A$ and $X$, is defined by setting
\begin{equation*}
H_{X,\,N}^A(\rn):=\lf\{f\in\cs'(\rn):\ \|M_N(f)\|_X<\infty\r\},
\end{equation*}
where $M_N(f)$ is the same as in \eqref{M_N}. Moreover,
for any $f\in H_{X,\,N}^A(\rn)$, let
$$\|f\|_{H_{X,\,N}^A(\rn)}:=\lf\|M_N(f)\r\|_X.$$
\end{definition}

Let $A$ be a dilation and $X$ the same as in Definition \ref{HXA}.
In the remainder of this article, we always let
\begin{align*}
N_{X,\,A}:=\lf\lfloor\lf(\frac{1}{\theta_0}-1\r)
\frac{\ln b}{\ln(\lz_-)}\r\rfloor+2.
\end{align*}

\begin{remark}\label{HXAre}
\begin{enumerate}
\item[(i)]
As was mentioned in \cite[Remark 2.17(i)]{wyy22},
the space $H_{X,\,N}^A(\rn)$ is independent of the choice
of $N$ as long as $N\in\nn\cap[N_{X,\,A},\infty)$.
Thus, in what follows, when $N\in\nn\cap[N_{X,\,A},\infty)$,
we always write $H_{X,\,N}^A(\rn)$ simply by $H_{X}^A(\rn)$.

\item[(ii)]
If $A:=2\,I_{n\times n}$, then $H_X^A(\rn)$ coincides with
$H_X(\rn)$ which was introduced by Sawano et al. in \cite{shyy17}.
\end{enumerate}
\end{remark}

\section{Fourier Transforms of $H_X^A(\rn)$ \label{s2}}
Let $A$ be a dilation, $X$ a ball quasi-Banach function space
satisfying some mild assumptions, and $f\in H_X^A(\rn)$.
In this section, we aim to study the Fourier transform of $f$.
Recall that, for any $\varphi\in\cs(\rn)$, its {\it Fourier transform},
denoted by $\mathscr{F}(\varphi)$ or $\widehat{\varphi}$,
is defined by setting, for any $\xi\in\rn$,
\begin{align*}
\mathscr{F}(\varphi)(\xi)=\widehat\varphi(\xi):=
\int_{\rn}\varphi(x)e^{-2\pi\imath x\cdot\xi}\,dx,
\end{align*}
here and thereafter, $\imath:=\sqrt{-1}$ and,
for any $x:=(x_1,\ldots,x_n),\xi:=(\xi_1,\ldots,\xi_n)\in\rn$,
$x\cdot\xi :=\sum_{i=1}^{n}x_i \xi_i$.
For any $f\in\cs'(\rn)$, $\widehat f$ is defined by setting,
for any $\varphi\in\mathcal{S}(\rn)$,
$\la\widehat f,\varphi\ra:=\la f,\widehat\varphi\ra$; also,
for any $f\in\mathcal{S}(\rn)$
[resp. $\mathcal{S}'(\rn)$],
$f^{\vee}$ denotes its \emph{inverse Fourier transform}
which is defined by setting,
for any $\xi\in\rn$, $f^{\vee}(\xi):=\widehat{f}(-\xi)$
[resp., for any $\varphi\in\mathcal{S}(\rn)$,
$\la f^{\vee},\varphi\ra:=\la f,\varphi^{\vee}\ra$].

Now, we present the main result of this section as follows.

\begin{theorem}\label{s2t1}
Let $A$ be a dilation and $X$ a ball quasi-Banach function space
satisfying both Assumption \ref{F-SVMI} with $p_-\in(0,\infty)$
and Assumption \ref{HMLonAsso} with the same
$p_-$, $\theta_0\in(0,\unp)$, and $p_0\in(\theta_0,\infty)$,
where $\unp$ is the same as in \eqref{underp}.
Further assume that there exists a $q_0\in[\theta_0,1]$ such that:
\begin{enumerate}
\item [{\rm(i)}]
for any non-negative measurable functions $\{f_k\}_{k=1}^\infty$,
\begin{align}\label{6.9.x1}
\sum_{k=1}^{{\infty}}\|f_k\|_{X^{\frac{1}{q_0}}}
&\ls\left\|\sum_{k=1}^{{\infty}}f_k\right\|_{X^{\frac{1}{q_0}}},
\end{align}
where the implicit positive constant is independent of $\{f_k\}_{k=1}^\infty$;

\item [{\rm(ii)}]
for any $B\in\CB$ with $\CB$ in \eqref{ball-B},
\begin{equation}\label{one_B_X}
\|\one_B\|_{X}\gs\min\lf\{|B|^{\frac{1}{q_0}},|B|^{\frac{1}{\theta_0}}\r\},
\end{equation}
where the implicit positive constant is independent of $B$.
\end{enumerate}
Then, for any $f\in H_X^A(\rn)$, there exists a
continuous function $F$ on $\rn$ such that
\begin{align}\label{6.10.x3}
\widehat{f}=F\quad\mbox{in}\quad\cs'(\rn)
\end{align}
and there exists a positive constant $C$, depending only on $A$ and $X$,
such that, for any $x\in\rn$,
\begin{align}\label{s2e1}
|F(x)|\le C\|f\|_{H_X^A(\rn)}\max\lf\{\lf[\rho_\ast(x)\r]^{\frac1{q_0}-1},
\lf[\rho_\ast(x)\r]^{\frac1{\theta_0}-1}\r\},
\end{align}
here and thereafter, $\rho_*$ is defined as in Definition \ref{def-shqn}
with $A$ replaced by its transposed matrix $A^*$.
\end{theorem}

\begin{remark}\label{s2re1}
\begin{enumerate}
\item [{\rm(i)}]
If $A:=2\,I_{n\times n}$, then Theorem \ref{s2t1}
was obtained in \cite[Theorem 2.1]{hcy22}.

\item [{\rm(ii)}]
For any given measurable set $E\st\rn$ and any given $p\in(0,\fz)$,
the \emph{Lebesgue space} $L^p(E)$ is defined by setting,
\begin{align}\label{6.11.y1}
L^p(E)&:=\left\{f\ \mbox{is measurable on}\ E:\
\|f\|_{L^p(E)}:=\left[\int_{E}|f(x)|^p\,dx\right]^{1/p}
<\infty\right\}.
\end{align}
Let $A$ be a dilation, $p\in(0,1)$, and
$$N\in\nn\cap\lf[\lf\lfloor\lf(\frac{1}{p}-1\r)
\frac{\ln b}{\ln(\lz_-)}\r\rfloor+2,\infty\r).$$
Then, by \cite[Remarks 2.7(i) and 4.21(i)]{yhyy},
we conclude that $L^p(\rn)$ satisfies all the
assumptions of Definition \ref{HXA} with $X:=L^p(\rn)$,
$p_{-}\in(0,p]$, $\tz_0\in(0,p_{-})$, and $p_0\in(p,\infty)$.
Moreover, choose $q_0\in(p,1]$. Then it follows from \eqref{6.11.y1}
that, for any non-negative measurable functions $\{f_k\}_{k=1}^\infty$
and any $B\in\CB$,
$$\sum_{k=1}^{{\infty}}\|f_k\|_{L^{\frac{p}{q_0}}(\rn)}
\le\left\|\sum_{k=1}^{{\infty}}f_k\right\|_{L^{\frac{p}{q_0}}(\rn)}$$
and
$$\|\one_B\|_{L^p(\rn)}=|B|^{\frac{1}{p}}
>\min\lf\{|B|^{\frac{1}{q_0}},|B|^{\frac{1}{\tz_0}}\r\}.$$
Thus, $L^p(\rn)$ satisfies all the assumptions
of Theorem \ref{s2t1} with $X:=L^p(\rn)$.
In this case, Theorem \ref{s2t1} was obtained in \cite[Theorem 1]{bw13}.

\item [{\rm(iii)}]
Recall that $\cp(\rn)$ is defined to be the
{\it set of all the measurable functions} $p(\cdot)$ on $\rn$ satisfying
\begin{align*}
0<\widetilde{p_-}:=\mathop\mathrm{ess\,inf}_{x\in\rn}p(x)\le
\mathop\mathrm{ess\,sup}_{x\in\rn}p(x)=:\widetilde{p_+}<\fz.
\end{align*}
For any $p(\cdot)\in \cp(\rn)$, the \emph{variable Lebesgue space}
$L^{p(\cdot)}(\rn)$ is defined to be the set of all the measurable
functions $f$ on $\rn$ such that
\begin{equation*}
\int_{\rn}|f(x)|^{p(x)}\,dx<\infty,
\end{equation*}
equipped with the \emph{quasi-norm}
$\|f\|_{L^{p(\cdot)}(\rn)}$ defined by setting
\begin{equation}\label{6.8.x1}
\|f\|_{L^{p(\cdot)}(\rn)}:=
\inf\lf\{\lambda\in(0,\infty):\  \int_{\rn}\lf[\frac{|f(x)|}
{\lambda}\r]^{p(x)}\,dx\leq 1\r\}.
\end{equation}
Denote by $C^{\log}(\rn)$ the set of all the functions $p(\cdot)\in\cp(\rn)$
satisfying the {\it globally log-H\"older continuous condition}, that is,
there exist $C_{\log}(p),C_\infty\in(0,\infty)$ and $p_\infty\in\rr$
such that, for any $x,y\in\rn$,
\begin{equation*}
|p(x)-p(y)|\leq\frac{C_{\log}(p)}{\log {(e+1/|x-y|)}}
\end{equation*}
and
\begin{equation*}
|p(x)-p_\infty|\leq\frac{C_\infty}{\log(e+|x|)}.
\end{equation*}
Let $A$ be a dilation and $p(\cdot)\in C^{\log}(\rn)$ satisfy
$0<\widetilde{p_-}\leq\widetilde{p_+}<1$ and
$$N\in\nn\cap\lf[\lf\lfloor\lf(\frac{1}{\widetilde{p_-}}-1\r)
\frac{\ln b}{\ln(\lz_-)}\r\rfloor+2,\infty\r).$$
Then, by \cite[Remarks 2.7(iv) and 4.21(v)]{yhyy},
we conclude that $L^{p(\cdot)}(\rn)$ satisfies all the
assumptions of Definition \ref{HXA} with $X:=L^{p(\cdot)}(\rn)$,
$p_{-}\in(0,\widetilde{p_-}]$, $\tz_0\in(0,p_{-})$,
and $p_0\in(\widetilde{p_+},\infty)$. Moreover,
choose $q_0\in(\widetilde{p_+},1]$. On the one hand,
from \cite[Remark 2.1(iv)]{yyyz16}, we deduce that,
for any non-negative measurable functions $\{f_k\}_{k=1}^\infty$,
$$\sum_{k=1}^{{\infty}}\|f_k\|_{L^{\frac{p(\cdot)}{q_0}}(\rn)}
\le\left\|\sum_{k=1}^{{\infty}}f_k\right\|
_{L^{\frac{p(\cdot)}{q_0}}(\rn)}.$$
On the other hand, by \eqref{6.8.x1}, we find that, for any $B\in\CB$,
$$\|\one_B\|_{L^{p(\cdot)}(\rn)}
\gs\min\lf\{|B|^{\frac{1}{\widetilde{p_+}}},|B|^{\frac{1}{\widetilde{p_-}}}\r\}
>\min\lf\{|B|^{\frac{1}{q_0}},|B|^{\frac{1}{\tz_0}}\r\}.$$
Thus, $L^{p(\cdot)}(\rn)$ satisfies all the assumptions
of Theorem \ref{s2t1} with $X:=L^{p(\cdot)}(\rn)$.
In this case, Theorem \ref{s2t1} was obtained in \cite[Theorem 1]{liu22}.

\item [{\rm(iv)}]
Let $\vec{p}:=(p_1,\ldots,p_n)\in(0,\fz]^n$.
Recall that the \emph{mixed-norm Lebesgue space $L^{\vec{p}}(\rn)$}
is defined to be the set of all the measurable
functions $f$ on $\rn$ such that
\begin{align}\label{6.8.x2}
\|f\|_{L^{\vec{p}}(\rn)}
:=\lf\{\int_{\mathbb{R}}\cdots
\lf[\int_{\mathbb{R}}
\lf\{\int_{\mathbb{R}}|f(x_1,\ldots,x_n)|^{p_1}\,dx_1\r\}
^{\frac{p_2}{p_1}}\,dx_2\r]^{\frac{p_3}{p_2}}\cdots\,dx_n\r\}
^{\frac{1}{p_n}}
\end{align}
is finite with the usual modifications made
when $p_i=\fz$ for some $i\in\{1,\ldots,n\}$.
Let $\vec{p}\in(0,1)^n$, $\widehat{p_-}:=\min\{p_1,...,p_n\}$,
$\widehat{p_+}:=\max\{p_1,...,p_n\}$, and
$$N\in\nn\cap\lf[\lf\lfloor\lf(\frac{1}{\widehat{p_-}}-1\r)
\frac{\ln b}{\ln(\lz_-)}\r\rfloor+2,\infty\r).$$
Then, by both \cite[p.\,2047]{zwyy} and \cite[Lemmas 7.22 and 7.26]{zwyy},
we conclude that $L^{\vec{p}}(\rn)$ satisfies all the
assumptions of Definition \ref{HXA} with $X:=L^{\vec{p}}(\rn)$,
$p_{-}\in(0,\widehat{p_-}]$, $\tz_0\in(0,p_-)$,
and $p_0\in(\widehat{p_+},\infty)$. Moreover,
choose $q_0\in(\widehat{p_+},1]$.
From \eqref{6.8.x2} and \cite[(9)]{llz22}, we deduce that,
for any non-negative measurable functions $\{f_k\}_{k=1}^\infty$
and any $B\in\CB$,
$$\sum_{k=1}^{{\infty}}\|f_k\|_{L^{\frac{\vec{p}}{q_0}}(\rn)}
\le C\left\|\sum_{k=1}^{{\infty}}f_k\right\|
_{L^{\frac{\vec{p}}{q_0}}(\rn)}$$
and
$$\|\one_B\|_{L^{\vec{p}}(\rn)}
\gs\min\lf\{|B|^{\frac{1}{\widehat{p_+}}},|B|^{\frac{1}{\widehat{p_-}}}\r\}
=\min\lf\{|B|^{\frac{1}{q_0}},|B|^{\frac{1}{\tz_0}}\r\}.$$
Thus, $L^{\vec{p}}(\rn)$ satisfies all the assumptions
of Theorem \ref{s2t1} with $X:=L^{\vec{p}}(\rn)$. In this case,
Theorem \ref{s2t1} was obtained in \cite[Theorem 3.1]{llz22}.

\item [{\rm(v)}]
As was mentioned in \cite[Remark 2.1(ii)]{hcy22},
\eqref{s2e1} implies that the function
$f\in H_X^A(\rn)\cap L^1(\rn)$ has a vanishing moment.
This illustrates the necessity of the vanishing moment
of atoms in some sense.
\end{enumerate}
\end{remark}

To prove Theorem \ref{s2t1}, we need more preparations.
Let $A$ be a dilation. Recall that the dilation operator
$D_A$ is defined by setting, for any $f\in\mathscr M(\rn)$,
$$D_A(f)(\cdot):=f(A\cdot).$$
Then, by an elementary calculation (see also \cite[(3.1)]{bw13}),
we find that, for any $k\in\zz$, $f\in L^1(\rn)$, and $x\in\rn$,
\begin{align}\label{6.10.x1}
\widehat{f}(x)&=b^k\lf(D_{A^*}^k\lf(\mathscr{F}\lf(D_{A}^kf\r)\r)\r)(x).
\end{align}

Next, we recall the definitions of anisotropic
$(X,q,d)$-atoms and anisotropic atomic Hardy spaces
$H_{X,{\rm atom}}^{A,q,d}(\rn)$ which were first introduced in
\cite[Definitions 4.1 and 4.2]{wyy22}.

\begin{definition}\label{atom}
Let $A$, $X$, $\tz_0$, and $p_0$ be the same as in Definition \ref{HXA}.
Further assume that $q\in(\max\{p_0,1\},\infty]$ and
\begin{equation}\label{def-d}
d\in\lf[\left\lfloor\left(\frac{1}{\theta_0}-1\right)
\frac{\ln b}{\ln(\lambda_-)}\right\rfloor,\fz\r)\cap\zz_+.
\end{equation}
\begin{enumerate}
\item [{\rm(i)}]
An {\it anisotropic $(X,q,d)$-atom} $a$ is a
measurable function on $\rn$ satisfying that
\begin{enumerate}
\item [${\rm(i)}_1$]
$\supp(a):=\{x\in\rn:\ a(x)\neq0\}\subset B$,
where $B\in\CB$ and $\CB$ is the same as in \eqref{ball-B};
\item[${\rm(i)}_2$]
$\|a\|_{L^q(\rn)}\leq |B|^{\frac{1}{q}}\|\one_B\|_X^{-1}$;
\item[${\rm(i)}_3$]
$\int_{\rn} a(x)x^\gamma\,dx=0 $ for any
$\gamma\in\zz_+^n$ with $|\gamma|\leq d$,
here and thereafter, for any $x:=(x_1,\ldots,x_n)\in\rn$
and $\gamma:=\{\gamma_1,\ldots,\gamma_n\}\in\zz_+^n$,
$|\gamma|:=\gamma_1+\cdots+\gamma_n$ and
$x^\gamma:=x_1^{\gamma_1}\cdots x_n^{\gamma_n}$.
\end{enumerate}
\item[{\rm(ii)}]
The {\it anisotropic atomic Hardy space} $\Hatom(\rn)$ is defined
to be the set of all the $f\in\cs'(\rn)$ satisfying that there exists
a sequence
$\{\lambda_j\}_{j\in\nn}\subset\cc$ and a sequence
$\{a_j\}_{j\in\nn}$ of $(X,q,d)$-atoms supported,
respectively, in $\{B^j\}_{j\in\nn}\subset\CB$ such that
\begin{equation*}
f=\sum_{j\in\nn}\lambda_j a_j
\end{equation*}
in $\cs'(\rn)$ and that
\begin{equation*}
\lf\|\left\lbrace\sum_{j\in\nn}
\left[\frac{|\lz_j|\one_{B^j}}{\|\one_{B^j}\|_X}\right]
^{\theta_0}\right\rbrace^{\frac{1}{\theta_0}}\r\|_X<\infty.
\end{equation*}
Moreover, for any $f\in\Hatom(\rn)$, let
\begin{equation*}
\|f\|_{\Hatom(\rn)}:=\inf\left\|\left\lbrace\sum_{j\in\nn}
\left[\frac{|\lz_j|\one_{B^j}}{\|\one_{B^j}\|_X}\right]^
{\theta_0}\right\rbrace^{\frac{1}{\theta_0}}\right\|_X,
\end{equation*}
where the infimum is taken over all the decompositions of $f$ as above.
\end{enumerate}
\end{definition}

The following atomic characterization of $H_X^A(\rn)$,
which was established in \cite[Theorem 4.3]{wyy22},
is needed in the proof of Theorem \ref{s2t1}.

\begin{lemma}\label{s2l0}
Let $A$, $X$, $q$, and $d$ be the same as in Definition \ref{atom}.
Then $H_X^A(\rn)=\Hatom(\rn)$ with equivalent quasi-norms.
\end{lemma}

By an argument similar to that used in proof of \cite[Lemma 4]{bw13},
we immediately obtain the following conclusioan.

\begin{lemma}\label{s2l1}
Let $A$, $X$, $q$, and $d$ be the same as in Definition \ref{atom}.
Assume that $a$ is an anisotropic $(X,q,d)$-atom supported in
$x_0+B_{i_0}$ with some $x_0\in\rn$ and $i_0\in\zz$.
Then there exists a positive constant $C$ such that,
for any $\alpha\in\zz_+^n$ with $|\az|\leq d$ and for
any $x\in\rn$,
\begin{align}\label{5.30.x2}
\lf|\partial^\az\lf(\mathscr{F}\lf(D_{A}^{i_0}a\r)\r)(x)\r|
&\leq C\lf\|\one_{B_{i_0}}\r\|_X^{-1}\min\lf\{1,|x|^{d-|\az|+1}\r\},
\end{align}
where $C$ is also independent of $a$.
\end{lemma}

\begin{proof}
Without loss of generality, we may assume that $a$ is supported in $B_{i_0}$.
Thus, $\supp(D_{A}^{i_0}a)\st B_0$. On the one hand,
by \cite[(1.20)]{duo}, Definition \ref{atom}${\rm (i)}_3$,
the Taylor remainder theorem, the H\"older inequality,
and Definition \ref{atom}${\rm (i)}_2$, we conclude that,
for any $\alpha\in\zz_+^n$ with $|\az|\leq d$ and for any $x\in\rn$,
\begin{align}\label{5.30.x1}
\lf|\partial^\az\lf(\mathscr{F}\lf(D_{A}^{i_0}a\r)\r)(x)\r|
&=\lf|\int_{B_0}(-2\pi\imath \xi)^{\az}\lf(D_{A}^{i_0}a\r)(\xi)
e^{-2\pi\imath x\cdot\xi}\,d\xi\r|\\
&=\lf|\int_{B_0}(-2\pi\imath \xi)^{\az}\lf(D_{A}^{i_0}a\r)(\xi)
\lf[e^{-2\pi\imath x\cdot\xi}-T(\xi)\r]\,d\xi\r|\noz\\
&\ls\int_{B_0}|\xi|^{|\az|}\lf|a\lf({A}^{i_0}\xi\r)\r|
|x|^{d-|\az|+1}|\xi|^{d-|\az|+1}\,d\xi\noz\\
&\ls|x|^{d-|\az|+1}b^{-i_0}\int_{B_{i_0}}|a\lf(\xi\r)|\,d\xi\noz\\
&\le|x|^{d-|\az|+1}\lf\|\one_{B_{i_0}}\r\|_X^{-1},\noz
\end{align}
where $T(\xi)$ is the $(d-|\az|)$th-order
Taylor polynomial of the function
$\xi\rightarrow e^{-2\pi\imath x\cdot\xi}$ at the origin.
On the other hand, from \cite[(1.20)]{duo}, the H\"older inequality,
and Definition \ref{atom}${\rm (i)}_2$, we deduce that,
for any $\alpha\in\zz_+^n$ with $|\az|\leq d$ and for any $x\in\rn$,
\begin{align*}
\lf|\partial^\az\lf(\mathscr{F}\lf(D_{A}^{i_0}a\r)\r)(x)\r|
&=\lf|\int_{B_0}(-2\pi\imath \xi)^{\az}\lf(D_{A}^{i_0}a\r)(\xi)
e^{-2\pi\imath x\cdot\xi}\,d\xi\r|\\
&\ls\int_{B_0}|\xi|^{|\az|}\lf|a\lf({A}^{i_0}\xi\r)\r|\,d\xi
\ls b^{-i_0}\int_{B_{i_0}}|a\lf(\xi\r)|\,d\xi\\
&\le\lf\|\one_{B_{i_0}}\r\|_X^{-1},
\end{align*}
which, combined with \eqref{5.30.x1},
further implies \eqref{5.30.x2} and hence
completes the proof of Lemma \ref{s2l1}.
\end{proof}

Applying Lemma \ref{s2l1}, we obtain the following uniform estimate
for anisotropic $(X,q,d)$-atoms, which plays a key role
in the proof
of Theorem \ref{s2t1}.
\begin{lemma}\label{s2l2}
Let $A$, $X$, $q$, $d$, and $\tz_0$ be the same as in Definition \ref{atom}.
Further assume that $X$ satisfies \eqref{one_B_X} with $q_0\in[\theta_0,1]$.
Then there exists a positive constant $C$ such that,
for any anisotropic $(X,q,d)$-atom $a$ and for any $x\in\rn$,
\begin{align}\label{s2e2}
\lf|\widehat{a}(x)\r|\leq C\max\lf\{\lf[\rho_*(x)\r]^{\frac1{q_0}-1},\,
\lf[\rho_*(x)\r]^{\frac1{\tz_0}-1}\r\},
\end{align}
where $\rho_*$ is the same as in Theorem \ref{s2t1}.
\end{lemma}

The proof of Lemma \ref{s2l2} needs the following inequalities
which are just \cite[p.\,11, Lemma 3.2]{Bownik}.

\begin{lemma}\label{s2l3}
Let $A$ be a dilation. Then there exists a positive
constant $C$ such that, for any $x\in\rn$,
\begin{equation*}
\frac{1}{C} [\rho(x)]^{\ln(\lambda_-)/\ln b} \leq \left|x\right|
\leq C[\rho(x)]^{\ln(\lambda_+)/\ln b}\ { when}\ \rho(x)\in(1,\fz)
\end{equation*}
and
\begin{equation*}
\frac{1}{C} [\rho(x)]^{\ln(\lambda_+)/\ln b} \leq \left|x\right|
\leq C[\rho(x)]^{\ln(\lambda_-)/\ln b}\ { when}\ \rho(x)\in[0,1].
\end{equation*}
\end{lemma}

Now, we give the proof of Lemma \ref{s2l2}.

\begin{proof}[Proof of Lemma \ref{s2l2}]
Let $a$ be an anisotropic $(X,q,d)$-atom supported in
$x_0+B_{i_0}$ with some $x_0\in\rn$ and $i_0\in\zz$.
Without loss of generality, we may assume $x_0=\mathbf{0}$.
By \eqref{6.10.x1}, Lemma \ref{s2l1} with $\az=(\overbrace{0,\ldots,0}^{n\ \rm times})$,
and
\eqref{one_B_X}, we conclude that,
for any $x\in\rn$,
\begin{align}\label{s2e3}
\lf|\widehat{a}(x)\r|
&=\lf|b^{i_0}\lf(D_{A^*}^{i_0}\lf(\mathscr{F}\lf(D_{A}^{i_0}a\r)\r)\r)(x)\r|
=\lf|b^{i_0}\mathscr{F}\lf(D_{A}^{i_0}a\r)\lf((A^*)^{i_0}x\r)\r|\\
&\ls b^{i_0}\lf\|\mathbf{1}_{B_{i_0}}\r\|_{X}^{-1}
\min\lf\{1,\lf|(A^*)^{i_0}x\r|^{d+1}\r\}\noz\\
&\ls b^{i_0}\max\lf\{b^{-\frac{i_0}{q_0}},\,b^{-\frac{i_0}{\tz_0}}\r\}
\min\lf\{1,\lf|(A^*)^{i_0}x\r|^{d+1}\r\}\noz.
\end{align}
Next, we prove \eqref{s2e2} by considering two cases:
$\rho_*(x)\leq b^{-i_0}$ and $\rho_*(x)>b^{-i_0}$.

Case 1) $\rho_{*}(x)\leq b^{-i_0}$. In this case, note that
\begin{align}\label{6.10.x2}
\rho_*\lf((A^*)^{i_0}x\r)&=b^{i_0}\rho_*(x)\leq 1.
\end{align}
Moreover, by \eqref{def-d}, we find that
$$1-\frac1{q_0}+(d+1)\frac{\ln(\lambda_-)}{\ln b}
\geq1-\frac1{\tz_0}+(d+1)\frac{\ln(\lambda_-)}{\ln b}>0.$$
From this, \eqref{s2e3}, \eqref{6.10.x2}, and Lemma \ref{s2l3},
we infer that, for any $x\in\rn$ satisfying $\rho_*(x)\leq b^{-i_0}$,
\begin{align}\label{s2e4}
\lf|\widehat{a}(x)\r|
&\ls b^{i_0}\max\lf\{b^{-\frac{i_0}{q_0}},\,b^{-\frac{i_0}{\tz_0}}\r\}
\lf[\rho_*\lf((A^*)^{i_0}x\r)\r]^{(d+1)\frac{\ln(\lambda_-)}{\ln b}}\\
&= \max\lf\{b^{i_0[1-\frac1{q_0}+(d+1)\frac{\ln(\lambda_-)}{\ln b}]},\,
b^{i_0[1-\frac1{\tz_0}+(d+1)\frac{\ln(\lambda_-)}{\ln b}]}\r\}
\lf[\rho_*(x)\r]^{(d+1)\frac{\ln(\lambda_-)}{\ln b}}\noz\\
&=\max\lf\{\lf[\rho_*(x)\r]^{\frac1{q_0}-1},\,
\lf[\rho_*(x)\r]^{\frac1{\tz_0}-1}\r\}.\noz
\end{align}
This shows \eqref{s2e2} in Case 1).

Case 2) $\rho_*(x)>b^{-i_0}$. In this case, note that
\begin{align*}
\rho_*\lf((A^*)^{i_0}x\r)&=b^{i_0}\rho_*(x)>1.
\end{align*}
Using this, \eqref{s2e3}, Lemma \ref{s2l3}, and the fact that
$$\frac1{\tz_0}-1\geq\frac1{q_0}-1\geq0,$$
we conclude that, for any $x\in\rn$ satisfying $\rho_*(x)>b^{-i_0}$,
\begin{align*}
\lf|\widehat{a}(x)\r|
&\ls b^{i_0}\max\lf\{b^{-\frac{i_0}{q_0}},\,b^{-\frac{i_0}{\tz_0}}\r\}
=\max\lf\{b^{-i_0(\frac1{q_0}-1)},\,b^{-i_0(\frac1{\tz_0}-1)}\r\}\\
&\leq\max\lf\{\lf[\rho_*(x)\r]^{\frac1{q_0}-1},\,\lf[\rho_*(x)\r]^{\frac1{\tz_0}-1}\r\},
\end{align*}
which, combined with \eqref{s2e4}, then
completes the proof of \eqref{s2e2} and hence Lemma \ref{s2l2}.
\end{proof}

The following inequality is basic and used throughout this article.

\begin{lemma}\label{basicine}
Let $\{a_j\}_{j\in\nn}\subset[0,\infty)$. If $\alpha\in(0,1]$, then
\begin{equation*}
\lf(\sum_{j\in\nn}a_j\r)^\alpha\leq\sum_{j\in\nn}a_j^\alpha.
\end{equation*}
\end{lemma}

The following conclusion is also used in the proof of Theorem \ref{s2t1}.

\begin{lemma}\label{s2l4}
Let $A$, $X$, and $\tz_0$ be the same as in Definition \ref{atom}.
Further assume that $X$ satisfies \eqref{6.9.x1} with $q_0\in[\theta_0,1]$.
Then there exists a positive constant $C$ such that,
for any $\{\lz_i\}_{i\in\nn}\subset\mathbb{C}$ and
$\{B^{(i)}\}_{i\in\nn}\subset\CB$,
$$\sum_{i\in\nn}|\lz_i|\le C\lf\|\lf\{\sum_{i\in\nn}
\lf[\frac{|\lz_i|\mathbf{1}_{B^{(i)}}}{\|\mathbf{1}_{B^{(i)}}\|_{X}}\r]^
{\tz_0}\r\}^{\frac{1}{\tz_0}}\r\|_{X}.$$
\end{lemma}

\begin{proof}
Indeed, by Lemma \ref{basicine}, Definition \ref{Debf},
\eqref{6.9.x1}, and Definition \ref{BQBFS}(ii), we find that,
for any $\{\lz_i\}_{i\in\nn}\subset\mathbb{C}$ and
$\{B^{(i)}\}_{i\in\nn}\subset\CB$,
\begin{align*}
\sum_{i=1}^{\fz}|\lz_i|
&\le \lf(\sum_{i=1}^{\fz}|\lz_i|^{q_0}\r)^{\frac{1}{q_0}}
=\lf\{\sum_{i=1}^{\fz}
\lf\|\frac{|\lz_i|\mathbf{1}_{B^{(i)}}}
{\|\mathbf{1}_{B^{(i)}}\|_{X}}\r\|_{X}^
{q_0}\r\}^{\frac{1}{q_0}}\\
&=\lf\{\sum_{i=1}^{\fz}
\lf\|\frac{|\lz_i|^{q_0}\mathbf{1}_{B^{(i)}}}{\|\mathbf{1}_{B^{(i)}}\|
_{X}^{q_0}}\r\|_{X^{\frac{1}{q_0}}}\r\}^{\frac{1}{q_0}}
\lesssim \lf\|\sum_{i=1}^{\fz}
\lf[\frac{|\lz_i|\mathbf{1}_{B^{(i)}}}{\|\mathbf{1}_{B^{(i)}}\|_{X}}\r]^
{q_0}\r\|_{X^{\frac{1}{q_0}}}^{\frac{1}{q_0}}\\
&=\lf\|\lf\{\sum_{i=1}^{\fz}\lf[
\frac{|\lz_i|\mathbf{1}_{B^{(i)}}}{\|\mathbf{1}_{B^{(i)}}\|_{X}}\r]^
{q_0}\r\}^{\frac{1}{q_0}}\r\|_{X}
\leq \lf\|\lf\{\sum_{i=1}^{\fz}\lf[
\frac{|\lz_i|\mathbf{1}_{B^{(i)}}}{\|\mathbf{1}_{B^{(i)}}\|_{X}}\r]^
{\tz_0}\r\}^{\frac{1}{\tz_0}}\r\|_{X}.
\end{align*}
This finishes the proof of Lemma \ref{s2l4}.
\end{proof}

Next, we show Theorem \ref{s2t1}.

\begin{proof}[Proof of Theorem \ref{s2t1}]
Let $q$ and $d$ be the same as in Definition \ref{atom}.
Without loss of generality, we may assume that $\|f\|_{H_X^A(\rn)}>0$.
Then, by Lemma \ref{s2l0} and Definition \ref{atom}(ii),
we find that there exists a sequence $\{\lz_i\}_{i\in\nn}\subset\mathbb{C}$
and a sequence $\{a_i\}_{i\in\nn}$ of anisotropic $(X,q,d)$-atoms supported,
respectively, in $\{B^{(i)}\}_{i\in\nn}\subset\CB$ such that
\begin{align}\label{s2e5}
f=\sum_{i\in\nn}\lz_ia_i
\ \mathrm{in} \ \cs'(\rn)
\end{align}
and
\begin{align}\label{s2e6}
\|f\|_{H_X^A(\rn)}\sim\lf\|\lf\{\sum_{i\in\nn}
\lf[\frac{|\lz_i|\mathbf{1}_{B^{(i)}}}{\|\mathbf{1}_{B^{(i)}}\|_{X}}\r]^
{\tz_0}\r\}^{1/\tz_0}\r\|_{X}.
\end{align}

First, we try to find the desired function $F$.
Note that a function $g\in L^1(\rn)$ implies that
$\widehat{g}$ is well defined in $\rn$
(see, for instance, \cite[(1.11)]{duo}),
so does $\widehat{a_i}$ for any $i\in\nn$.
Moreover, from Lemmas \ref{s2l2} and \ref{s2l4}
and from \eqref{s2e6}, it follows that, for any $x\in\rn$,
\begin{align}\label{s2e8}
\sum_{i\in\nn}|\lz_i||\widehat{a_i}(x)|
&\ls\sum_{i\in\nn}|\lz_i|
\max\lf\{\lf[\rho_*(x)\r]^{\frac1{q_0}-1},\,
\lf[\rho_*(x)\r]^{\frac1{\tz_0}-1}\r\}\\
&\ls\|f\|_{H_X^A(\rn)}
\max\lf\{\lf[\rho_*(x)\r]^{\frac1{q_0}-1},\,
\lf[\rho_*(x)\r]^{\frac1{\tz_0}-1}\r\}<\fz.\noz
\end{align}
Therefore, the function
\begin{align}\label{s2e9}
F(\cdot):=\sum_{i\in\nn}\lz_i\widehat{a_i}(\cdot)
\end{align}
is well defined pointwisely on $\rn$ and, for any $x\in\rn$,
\begin{align*}
|F(x)|\ls\|f\|_{H_X^A(\rn)}
\max\lf\{\lf[\rho_*(x)\r]^{\frac1{q_0}-1},\,
\lf[\rho_*(x)\r]^{\frac1{\tz_0}-1}\r\},
\end{align*}
which completes the proof of \eqref{s2e1}.

Second, we show the continuity of $F$ on $\rn$.
If we can prove that $F$ is continuous on any compact subset of $\rn$,
then the continuity of $F$ on $\rn$ is obvious.
Let $E\subset\rn$ be any given compact set.
Then there exists a positive constant $K$,
depending only on $A$ and $E$,
such that $\rho_*(x)\le K$ holds true for any $x\in E$.
By this and \eqref{s2e8}, we conclude that, for any $x\in E$,
\begin{align*}
\sum_{i\in\nn}|\lz_i||\widehat{a_i}(x)|
&\ls\max\lf\{K^{\frac1{q_0}-1},\,K^{\frac1{\tz_0}-1}\r\}
\|f\|_{H_X^A(\rn)}<\fz.
\end{align*}
Thus, the summation $\sum_{i\in\nn}\lz_i\widehat{a_i}(\cdot)$
converges uniformly on $E$. This, together with the fact that
$\widehat{a_i}$ is continuous for any $i\in\nn$, further
implies that $F$ is also continuous on $E$ and hence on $\rn$.

Finally, we show \eqref{6.10.x3}.
By \eqref{s2e5} and the continuity of the Fourier transform in $\cs'(\rn)$
(see, for instance, \cite[Theorem 1.17]{duo}), we obtain
\begin{align*}
\widehat{f}=\sum_{i\in\nn}\lz_i\widehat{a_i}
\ \mathrm{in}\ \cs'(\rn).
\end{align*}
Thus, to prove \eqref{6.10.x3}, we only need to show that
\begin{align}\label{s2e10}
F=\sum_{i\in\nn}\lz_i\widehat{a_i}
\ \mbox{in}\ \cs'(\rn).
\end{align}
Indeed, from Lemma \ref{s2l2} and the definition of Schwartz functions
[see \eqref{6.11.x1}], we deduce that,
for any $i\in\nn$ and $\varphi\in\cs(\rn)$,
\begin{align*}
|\langle \widehat{a_i},\varphi\rangle|
&=\lf|\int_{\rn}\widehat{a_i}(x)\varphi(x)\,dx\r|\\
&\leq\sum_{k=1}^\fz\int_{(A^*)^{k+1}B^*_0\setminus(A^*)^{k}B^*_0}
\max\lf\{\lf[\rho_*(x)\r]^{\frac1{q_0}-1},\,\lf[\rho_*(x)\r]^{\frac1{\tz_0}-1}\r\}
|\varphi(x)|\,dx\\
&\quad+\|\varphi\|_{L^1(\rn)}\\
&\ls\sum_{k=1}^\fz b^{k+1}b^{k(\frac1{\tz_0}-1)}
b^{-k(\lfloor\frac1{\tz_0}\rfloor+2)}
+\|\varphi\|_{L^1(\rn)}\\
&\le\sum_{k=1}^\fz b^{-k}+\|\varphi\|_{L^1(\rn)},
\end{align*}
where $B^*_0$ is the unit dilated ball with respect to $A^*$.
This further
implies that there exists a positive constant $C$ such that
$|\langle \widehat{a_i},\varphi\rangle|\le C$ holds true uniformly
for any $i\in\nn$. Combining this and \eqref{s2e6}, we have
\begin{align*}
\lim_{I\to\fz}\sum_{i=I+1}^\fz|\lz_i||\langle \widehat{a_i},\varphi\rangle|
\ls \lim_{I\to\fz}\sum_{i=I+1}^\fz|\lz_i|=0.
\end{align*}
Therefore, for any $\varphi\in\cs(\rn)$,
$$\langle F,\varphi\rangle
=\lim_{I\to\fz}\lf\langle\sum_{i=1}^I\lz_i\widehat{a_i},\,
\varphi\r\rangle.$$
This finishes the proof of \eqref{s2e10} and hence Theorem \ref{s2t1}.
\end{proof}

\section{Hardy--Littlewood Inequalities on $H_X^A(\rn)$ \label{s3}}
In this section, as applications of Theorem \ref{s2t1},
we first prove that
the function $F$ given in Theorem \ref{s2t1} has a
higher order convergence at the origin (see Theorem \ref{s3t1} below).
Then we extend the Hardy--Littlewood inequality to
the setting of anisotropic Hardy spaces associated
with ball quasi-Banach function spaces
(see Theorem \ref{s3t2} below). In what follows, $\varepsilon\to0^+$
means that there exists an $\alpha_0\in(0,\infty)$ such that $\varepsilon
\in(0,\alpha_0)$ and $\varepsilon\to 0$.

\begin{theorem}\label{s3t1}
Let $A$, $X$, $q_0$, and $\rho_*$ be the same as in Theorem \ref{s2t1}.
Then, for any $f\in H_X^A(\rn)$, there exists a continuous function $F$ on $\rn$
such that
$\widehat{f}=F$ in $\cs'(\rn)$
and
\begin{align}\label{s3e1}
\lim_{|x|\to0^+}\frac{F(x)}{[\rho_*(x)]^{\frac1{\tz_0}{-1}}}=0.
\end{align}
\end{theorem}

\begin{proof}
Let $f\in H_X^A(\rn)$ and $q$ and $d$ be the same as in Definition \ref{atom}.
Then, by Lemma \ref{s2l0} and Definition \ref{atom}(ii), we find that
there exists
a sequence $\{\lz_i\}_{i\in\nn}\subset\mathbb{C}$ and a
sequence $\{a_i\}_{i\in\nn}$ of anisotropic $(X,q,d)$-atoms
supported, respectively, in $\{B^{(i)}\}_{i\in\nn}\subset\CB$
such that
\begin{align*}
f=\sum_{i\in\nn}\lz_ia_i
\ \mbox{in}\ \cs'(\rn)
\end{align*}
and
\begin{align}\label{6.11.x2}
\|f\|_{H_X^A(\rn)}&\sim\lf\|\lf\{\sum_{i\in\nn}
\lf[\frac{|\lz_i|\mathbf{1}_{B^{(i)}}}{\|\mathbf{1}_{B^{(i)}}\|_{X}}\r]^
{\tz_0}\r\}^{\frac{1}{\tz_0}}\r\|_{X}.
\end{align}
Moreover, from the proof of Theorem \ref{s2t1},
it follows that, for any $x\in\rn$,
\begin{align}\label{s3e3}
F(x)=\sum_{i\in\nn}\lz_i\widehat{a_i}(x)
\end{align}
is continuous and satisfies that $\widehat{f}=F$ in $\cs'(\rn)$.
Thus, to show the present theorem, we only need to prove that
\eqref{s3e1} holds true for $F$ in \eqref{s3e3}.
On the one hand, by an argument similar to that used in the
proof of \eqref{s2e4}, we conclude that, for any anisotropic
$(X,q,d)$-atom $a$ supported in $x_0+B_{k_0}$ with some
$x_0\in\rn$ and $k_0\in\zz$ and for any $x\in\rn$ satisfying
$\rho_*(x)\leq b^{-k_0}$,
$$\lf|\widehat{a}(x)\r|\lesssim
\max\lf\{b^{k_0[1-\frac1{q_0}+(d+1)\frac{\ln(\lambda_-)}{\ln b}]},\,
b^{k_0[1-\frac1{\tz_0}+(d+1)\frac{\ln(\lambda_-)}{\ln b}]}\r\}
\lf[\rho_*(x)\r]^{(d+1)\frac{\ln(\lambda_-)}{\ln b}},$$
which, together with \eqref{def-d} and Lemma \ref{s2l3},
further implies that
$$(d+1)\frac{\ln(\lambda_-)}{\ln b}>\frac1{\tz_0}-1$$
and
\begin{align}\label{s3e4}
\lim_{|x|\to0^+}\frac{|\widehat{a}(x)|}{[\rho_*(x)]^{\frac1{\tz_0}{-1}}}=0.
\end{align}
On the other hand, from \eqref{s3e3}, Lemmas \ref{s2l2}, \ref{s2l3}, and \ref{s2l4},
and \eqref{6.11.x2}, we deduce that, for any $x\in\rn$ satisfying $|x|<1$,
\begin{align}\label{s3e5}
\frac{|F(x)|}{[\rho_*(x)]^{\frac1{\tz_0}{-1}}}
&\leq\sum_{i\in\nn}|\lz_i|\frac{|\widehat{a_i}(x)|}
{[\rho_*(x)]^{\frac1{\tz_0}{-1}}}
\ls\sum_{i\in\nn}|\lz_i|\ls\|f\|_{H_X^A(\rn)}<\fz.
\end{align}
Using this, the dominated convergence theorem,
and \eqref{s3e4}, we find that
\begin{align*}
\lim_{|x|\to0^+}\frac{F(x)}{[\rho_*(x)]^{\frac1{q_0}{-1}}}=0,
\end{align*}
which completes the proof of Theorem \ref{s3t1}.
\end{proof}

As another application of Theorem \ref{s2t1},
we extend the Hardy--Littlewood inequality to the setting of anisotropic
Hardy spaces associated with ball quasi-Banach function spaces as follows.

\begin{theorem}\label{s3t2}
Let $A$, $X$, $\tz_0$, and $q_0$ be the same as in Theorem \ref{s2t1}.
Then, for any $f\in H_X^A(\rn)$, there exists a continuous function
$F$ on $\rn$ such that $\widehat{f}=F$ in $\cs'(\rn)$ and
\begin{align}\label{s3e6}
\lf[\int_\rn|F(x)|^{q_0}\min\lf\{\lf[\rho_*(x)\r]^
{q_0-\frac{q_0}{\tz_0}-1},\,
\lf[\rho_*(x)\r]^{q_0-2}\r\}\,dx\r]^{\frac1{q_0}}
\leq C\|f\|_{{H_X^A(\rn)}},
\end{align}
where $C$ is a positive constant depending only on $A$ and $X$.
\end{theorem}

\begin{proof}
Let $p_0$ and $d$ be the same as in Definition \ref{atom},
$q\in(\max\{p_0,2\},\fz]$, and $f\in H_X^A(\rn)$.
Then, by Lemma \ref{s2l0} and Definition \ref{atom},
we find that there exists a sequence $\{\lz_i\}_{i\in\nn}\subset\mathbb{C}$
and a sequence $\{a_i\}_{i\in\nn}$ of $(X,q,d)$-atoms supported,
respectively, in $\{B^{(i)}\}_{i\in\nn}\subset\CB$ such that
\begin{align*}
f=\sum_{i\in\nn}\lz_ia_i
\ \mbox{in}\ \cs'(\rn)
\end{align*}
and
\begin{align}\label{s3e7}
\lf\|\lf\{\sum_{i\in\nn}
\lf[\frac{|\lz_i|\mathbf{1}_{B^{(i)}}}{\|\mathbf{1}_{B^{(i)}}\|_{{X}}}\r]^
{\tz_0}\r\}^{1/\tz_0}\r\|_{{X}}\sim\|f\|_{{H_X^A(\rn)}}<\fz.
\end{align}
By Theorem \ref{s2t1}, we find that, to prove the present theorem,
it suffices to show that \eqref{s3e6} holds true for $F$ in \eqref{s2e9}.
For this purpose, we first prove that there exists a positive constant
$M$ such that, for any $(X,q,d)$-atom $a$, it holds true that
\begin{align}\label{s3e10}
\left(\int_{{\mathbb{R}^n}}\left[\lf|\widehat{a}(x)\r|
\min\left\{\left[\rho_{*}(x)\right]^{1-\frac 1{\tz_0}-\frac 1{q_0}},
\,\left[\rho_{*}(x)\right]^{1-\frac 2{q_0}}
\right\}\right]^{q_0}\,dx\right)^{\frac{1}{q_0}}\le M.
\end{align}
Without loss of generality, we may assume that $a$ is
supported in $x_0+B_{i_0}$ with some $x_0\in\rn$ and $i_0\in\zz$.
Then it is easy to conclude that
\begin{align}\label{6.11.z1}
&\left\{\int_{{\mathbb{R}^n}}\left[\lf|\widehat{a}(x)\r|
\min\left\{\left[\rho_{*}(x)\right]^{1-\frac 1{\tz_0}-\frac 1{q_0}},
\,\left[\rho_{*}(x)\right]^{1-\frac 2{q_0}}\right\}
\right]^{q_0}\,dx\right\}^{\frac{1}{q_0}}\\
&\quad\ls\left\{\int_{(A^*)^{-i_0+1}B_0^*}\left[\lf|\widehat{a}(x)\r|
\min\left\{\left[\rho_{*}(x)\right]^{1-\frac 1{\tz_0}-\frac 1{q_0}},
\,\left[\rho_{*}(x)\right]^{1-\frac 2{q_0}}\right\}
\right]^{q_0}\,dx\right\}^{\frac{1}{q_0}}\noz\\
&\qquad+\left\{\int_{((A^*)^{-i_0+1}B_0^*)^{\complement}}
\cdots\,dx\right\}^{\frac{1}{q_0}}\noz\\
&\quad=:{\rm I}_1+{\rm I}_2,\noz
\end{align}
where $B^*_0$ is the unit dilated ball with respect to $A^*$.
Let $\theta$ be a fixed positive constant such that
$$1-\frac1{q_0}+(d+1)\frac{\ln(\lambda_-)}{\ln b}-\theta
\geq1-\frac1{\tz_0}+(d+1)\frac{\ln(\lambda_-)}{\ln b}-\theta>0.$$
Using this and  \eqref{s2e4}, we find that
\begin{align}\label{6.11.z2}
{\rm I}_1&\ls b^{i_0[1+(d+1)\frac{\ln\lambda_-}{\ln b}]}
\max\lf\{b^{-\frac{i_0}{\tz_0}},\,b^{-\frac{i_0}{q_0}}\r\}\\
&\quad\times\left\{\int_{(A^*)^{-i_0+1}B_0^*}\left[
\min\left\{\left[\rho_{*}(x)\right]^{1-\frac 1{\tz_0}-\frac 1{q_0}+(d+1)
\frac{\ln\lambda_-}{\ln b}},\noz\r.\r.\r.\\
&\quad\lf.\lf.\lf.
\left[\rho_{*}(x)\right]^{1-\frac 2{q_0}+(d+1)\frac{\ln\lambda_-}{\ln b}}
\right\}\right]^{q_0}\,dx\right\}^{\frac{1}{q_0}}\noz\\
&\leq b^{i_0[1+(d+1)\frac{\ln\lambda_-}{\ln b}]}
\max\lf\{b^{-\frac{i_0}{\tz_0}},\,b^{-\frac{i_0}{q_0}}\r\}\noz\\
&\quad\times\min\lf\{b^{-i_0[1-\frac1{\tz_0}+(d+1)
\frac{\ln\lambda_-}{\ln b}-\theta]},\,b^{-i_0[1-\frac1{q_0}+(d+1)
\frac{\ln\lambda_-}{\ln b}-\theta]}\r\}\noz\\
&\quad \times\left\{\int_{(A^*)^{-i_0+1}B_0^*}
\left[\rho_{*}(x)\right]^{\theta q_0-1}\,dx\right\}
^{\frac{1}{q_0}}\noz\\
&= b^{i_0\theta}\left[\sum_{k\in\zz\setminus\nn}b^{-i_0+k}(b-1)
b^{(-i_0+k)(\theta q_0-1)}\r]^{\frac{1}{q_0}}
=\lf(\frac{b-1}{1-b^{-\theta q_0}}\r)^{\frac1{q_0}}.\noz
\end{align}
For ${\rm I}_2$, by the H\"{o}lder inequality, the Plancherel theorem
(see \cite[Theorem 1.18]{duo}), $q_0\in[\tz_0,1]$,
Definition \ref{atom}${\rm (i)}_2$, and \eqref{one_B_X}, we obtain
\begin{align*}
{\rm I}_2
&\leq\left\{\int_{((A^*)^{-i_0+1}B_0^*)^{\complement}}
\lf|\widehat{a}(x)\r|^2\,dx\right\}^{\frac12}\\
&\quad\times\left\{\int_{((A^*)^{-i_0+1}B_0^*)^{\complement}}
\left[\min\left\{\left[\rho_{*}(x)\right]^{1-\frac 1{p_-}-\frac 1{q_0}},\,
\left[\rho_{*}(x)\right]^{1-\frac 2{q_0}}\right\}\right]
^{\frac{2q_0}{2-q_0}}\,dx\right\}^{\frac{2-q_0}{2q_0}}\\
&\leq\|a\|_{L^2({\mathbb{R}^n})}\left\{\sum_{k\in\nn} b^{-i_0+k}(b-1)
\left[\min\left\{b^{(-i_0+k)(1-\frac 1{p_-}-\frac 1{q_0})},\,
b^{(-i_0+k)(1-\frac 2{q_0})}\right\}\right]^{\frac{2q_0}{2-q_0}}\right\}
^{\frac{2-q_0}{2q_0}}\\
&\leq\|a\|_{L^2({\mathbb{R}^n})}\left\{b^{-i_0}
\left[\min\left\{b^{-i_0(1-\frac 1{p_-}-\frac 1{q_0})},\,
b^{-i_0(1-\frac 2{q_0})}\right\}\right]^{\frac{2q_0}{2-q_0}}\right\}
^{\frac{2-q_0}{2q_0}}\\
&\ls\max\left\{b^{i_0(\frac12-\frac 1{p_-})},\,
b^{i_0(\frac12-\frac 1{q_0})}\right\}
\min\left\{b^{-i_0(\frac12-\frac 1{p_-})},\,
b^{-i_0(\frac12-\frac 1{q_0})}\right\}\\
&=1,
\end{align*}
which, together with \eqref{6.11.z1} and \eqref{6.11.z2},
further implies \eqref{s3e10}.

Next, we prove \eqref{s3e6}. From Lemma \ref{basicine},
Definition \ref{Debf}, \eqref{6.9.x1}, Definition \ref{BQBFS}(ii),
an argument similar to that used in the proof of Lemma \ref{s2l4},
and \eqref{s3e7}, we deduce that
\begin{align*}
\lf(\sum_{i=1}^{\fz}|\lz_i|^{q_0}\r)^{\frac{1}{q_0}}
\ls\|f\|_{{H_X^A(\rn)}}.
\end{align*}
By this, \eqref{s2e9}, $q_0\in[\tz_0,1]$, Lemma \ref{basicine},
the Fatou lemma, and \eqref{s3e10}, we conclude that
\begin{align*}
&\lf[\int_{{\mathbb{R}^n}}|F(x)|^{q_0}
\min\left\{\left[\rho_{*}(x)\right]^{q_0-\frac {q_0}{\tz_0}-1},\,
\left[\rho_{*}(x)\right]^{q_0-2}\right\}\,dx\r]^{\frac1{q_0}}\\
&\quad\le\lf\{\sum_{i\in{\mathbb N}}|\lambda_i|^{q_0}
\int_{{\mathbb{R}^n}}\left[\lf|\widehat{a_i}(x)\r|
\min\left\{\left[\rho_{*}(x)\right]^{1-\frac 1{\tz_0}-\frac 1{q_0}},\,
\left[\rho_{*}(x)\right]^{1-\frac 2{q_0}}\right\}\right]^{q_0}\,dx\r\}^{\frac1{q_0}}\noz\\
&\quad\ls M\left(\sum_{i\in{\mathbb N}}|\lambda_i|^{q_0}\right)^{\frac{1}{q_0}}
\ls\|f\|_{{H_X^A(\rn)}}.\noz
\end{align*}
This finishes the proof of Theorem \ref{s3t2}.
\end{proof}

\begin{remark}\label{6.11rem}
\begin{enumerate}
\item [{\rm(i)}]
If $A:=2\,I_{n\times n}$, then Theorems \ref{s3t1} and \ref{s3t2}
were obtained, respectively, in \cite[Theorems 2.2 and 2.3]{hcy22}.

\item [{\rm(ii)}]
Let $A$ be a dilation and $p\in(0,1)$. Then,
by Remark \ref{s2re1}(ii), we find that $L^p(\rn)$ satisfies all
the assumptions of Theorems \ref{s3t1} and \ref{s3t2} with $X:=L^{p}(\rn)$.
In this case, Theorems \ref{s3t1} and \ref{s3t2} were obtained,
respectively, in \cite[Corollaries 6 and 8]{bw13}.

\item [{\rm(iii)}]
Let $A$ be a dilation and $p(\cdot)\in C^{\log}(\rn)$ satisfy
$0<\widetilde{p_-}\leq\widetilde{p_+}<1$.
Then, by Remark \ref{s2re1}(iii),
we conclude that $L^{p(\cdot)}(\rn)$ satisfies all the assumptions
of Theorems \ref{s3t1} and \ref{s3t2} with $X:=L^{p(\cdot)}(\rn)$.
In this case, Theorems \ref{s3t1} and \ref{s3t2} were obtained,
respectively, in \cite[Theorems 2 and 3]{liu22}.

\item [{\rm(iv)}]
Let $\vec{p}\in(0,1)^n$.
Then, by Remark \ref{s2re1}(iv),
we conclude that $L^{\vec{p}}(\rn)$ satisfies all the assumptions
of Theorems \ref{s3t1} and \ref{s3t2} with $X:=L^{\vec{p}}(\rn)$.
In this case, Theorems \ref{s3t1} and \ref{s3t2} were obtained,
respectively, in \cite[Theorems 4.1 and 4.3]{llz22}.
\end{enumerate}
\end{remark}

\section{Several  Applications \label{s4}}
In this section, we apply Theorems \ref{s2t1},
\ref{s3t1}, and \ref{s3t2} to five concrete examples
of ball quasi-Banach function spaces, namely
Morrey spaces (see Subsection \ref{s6-appl1} below),
Lorentz spaces (see Subsection \ref{s6-appl3} below),
Orlicz spaces (see Subsection \ref{s6-appl7} below),
Orlicz-slice spaces (see Subsection \ref{s6-appl2} below),
and local generalized Herz--Hardy spaces
(see Subsection \ref{s6-appl6} below).

\subsection{Morrey Spaces\label{s6-appl1}}

Recall that the classical Morrey space
$M_{q}^{p}(\mathbb{R}^{n})$ with $0<q\leq p<\infty$,
originally introduced by Morrey \cite{morrey38} in 1938,
plays a key role in harmonic analysis
and partial differential equations. Since then,
various variants of Morrey spaces over different
underlying spaces have been investigated and developed
(see, for instance, \cite{cltl20,ho15,ho17,ho23no1,ho23no2,st09}).

\begin{definition}\label{6.11def}
Let $ 0<q\leq p<\infty $.
The {\it anisotropic Morrey space} $M_{q,A}^p(\rn)$ is defined
to be the set of all the measurable functions $f$ on $\rn$ such that
\begin{equation*}
\|f\|_{M_{q,A}^p(\rn)}:=\sup_{B\in\CB}
|B|^{1/p-1/q}\|f\|_{L^q(B)}<\infty,
\end{equation*}
where $\CB$ is the same as in \eqref{ball-B}.
\end{definition}

If $0<q\leq p<1$, then, obviously, $M_{q,A}^p(\rn)$
is a ball quasi-Banach function space.
From these and \cite[Remark 8.4]{wyy22},
we deduce that $M_{q,A}^{p}(\rn)$ satisfies all the assumptions
of Definition \ref{HXA} with $X:=M_{q,A}^{p}(\rn)$, $p_{-}\in(0,q]$,
$\tz_0\in(0,p_-)$, and $p_0\in(p,\infty)$.
Moreover, choose a $q_0\in(p,1]$. Then, from Definition \ref{6.11def},
we infer that, for any non-negative measurable functions
$\{f_k\}_{k=1}^\infty$ and any $B\in\CB$,
$$\sum_{k=1}^{{\infty}}\|f_k\|_{M_{q/q_0,A}^{p/q_0}(\rn)}
\le\left\|\sum_{k=1}^{{\infty}}f_k\right\|
_{M_{q/q_0,A}^{p/q_0}(\rn)}$$
and
$$\|\one_B\|_{M_{q,A}^p(\rn)}\ge|B|^{\frac{1}{p}}
>\min\lf\{|B|^{\frac{1}{q_0}},|B|^{\frac{1}{\tz_0}}\r\}.$$
Therefore, all the assumptions of Theorems \ref{s2t1}, \ref{s3t1},
and \ref{s3t2} are satisfied with $X:=M_{q,A}^{p}(\rn)$.
Applying Theorems \ref{s2t1}, \ref{s3t1},
and \ref{s3t2}, we obtain the following conclusion.

\begin{theorem}\label{Thsm}
If $\ 0<q\leq p<1$, then Theorems \ref{s2t1},
\ref{s3t1}, and \ref{s3t2} still hold true with $X$
replaced by $M_{q,A}^{p}(\rn)$.
\end{theorem}

\begin{remark}\label{rern}
We point out that Theorem \ref{Thsm}
even when $A:=2I_{n\times n}$
is completely new.
\end{remark}

\subsection{Lorentz Spaces\label{s6-appl3}}
Let $p\in(0,\fz)$ and $q\in(0,\fz)$.
Recall that the {\it Lorentz space} $L^{p,q}(\rn)$
is defined to be the set of all the measurable functions
$f$ on $\rn$ with the following
finite \emph{quasi-norm}
\begin{align}\label{6.13.x1}
\|f\|_{L^{p,q}(\rn)}:=
\begin{cases}
\lf[\displaystyle\frac{q}{p}\displaystyle\int_0^{\infty}\lf\{
t^{\frac{1}{p}}f^*(t)\r\}^q\frac{dt}{t}\r]^{\frac{1}{q}}
\ &{\rm if}\ q\in(0,\infty),\\
\displaystyle{\sup_{t\in(0,\infty)}}
\lf[t^{\frac{1}{p}}f^*(t)\r]&{\rm if}\ q=\infty
\end{cases}
\end{align}
with the usual modification made when $p=\fz$,
where $f^*$ denotes the \emph{non-increasing rearrangement} of $f$,
that is, for any $t\in(0,\infty)$,
\begin{equation*}
f^*(t):=\inf\lf\{\alpha\in(0,\infty):\ d_f(\alpha)\leq t\r\}
\end{equation*}
with $d_f(\alpha):=|\{x\in\rn:\ |f(x)|>\alpha\}|$
for any $\az\in(0,\fz)$.

Let $p\in(0,1)$, $q\in(0,1)$, and
\begin{align*}
N\in\nn\cap\lf[\displaystyle\lf\lfloor\lf(\frac{1}{p}-1\r)
\frac{\ln b}{\ln(\lambda_-)}\r\rfloor+2,\infty\r).
\end{align*}
Then, by \cite[Remarks 2.7(ii) and 4.21(ii)]{yhyy},
we conclude that $L^{p,q}(\rn)$ satisfies all the assumptions of
Definition \ref{HXA} with $X:=L^{p,q}(\rn)$, $p_{-}\in(0,\min\{p,q\}]$,
$\tz_0\in(0,p_{-})$, and $p_0\in(\max\{p,q\},\infty)$.
Moreover, choose a $q_0\in(\max\{p,q\},1]$.
From \eqref{6.13.x1}, we further deduce that, for any non-negative
measurable functions $\{f_k\}_{k=1}^\infty$ and any $B\in\CB$,
$$\sum_{k=1}^{{\infty}}\|f_k\|_{L^{p/q_0,q/q_0}(\rn)}
\le\left\|\sum_{k=1}^{{\infty}}f_k\right\|_{L^{p/q_0,q/q_0}(\rn)}$$
and
$$\|\one_B\|_{L^{p,q}(\rn)}=\lf\{\frac{q}{p}\displaystyle\int_0^{|B|}
t^{\frac{q}{p}-1}dt\r\}^{\frac{1}{q}}=|B|^{\frac{1}{p}}
\ge\min\lf\{|B|^{\frac{1}{q_0}},|B|^{\frac{1}{\tz_0}}\r\}.$$
Therefore, all the assumptions of Theorems \ref{s2t1}, \ref{s3t1},
and \ref{s3t2} are satisfied with $X:=L^{p,q}(\rn)$.
Applying Theorems \ref{s2t1}, \ref{s3t1},
and \ref{s3t2}, we obtain the following conclusion.

\begin{theorem}\label{thlorentz}
If $p\in(0,1)$ and $q\in(0,1)$, then Theorems \ref{s2t1}, \ref{s3t1},
and \ref{s3t2} still hold true with $X$ replaced by $L^{p,q}(\rn)$.
\end{theorem}

\begin{remark}\label{remrentz}
We point out that Theorem \ref{thlorentz}
even when $A:=2I_{n\times n}$ is completely new.
\end{remark}

\subsection{Orlicz Spaces\label{s6-appl7}}	
Recall that a function $\Phi:[0,\infty)\to[0,\infty)$
is called an {\it Orlicz function} if it is non-decreasing,
$\Phi(0)=0$, $\Phi(t)>0$ for any $t\in(0,\infty)$,
and $\lim_{t\to\infty}\Phi(t)=\infty$.
The function $ \Phi $ is said to be of {\it upper}
(resp. {\it lower}) {\it type} $p$ for some $p\in[0,\infty)$
if there exists a positive constant $C$ such that,
for any $s\in[1,\infty)$ (resp. $s\in[0,1]$)
and $t\in[0,\infty)$, $\Phi(st)\leq Cs^p\Phi(t)$.
The \emph{Orlicz space} $L^{\Phi}(\rn)$ is defined to be
the set of all the measurable functions $f$ on $\rn$ such that
$$\|f\|_{L^{\Phi}(\rn)}:=\inf\lf\{\lz\in(0,\fz):\
\int_{\rn}\Phi\lf(\frac{|f(x)|}{\lz}\r)\,dx\le1\r\}<\infty.$$

Let $\Phi$ be an Orlicz function with lower type $p_{\Phi}^-$ and
upper type $p_{\Phi}^+$ satisfying $0<p_{\Phi}^-\le p_{\Phi}^+<1$
and let
\begin{align*}
N\in\nn\cap\lf[\displaystyle\lf\lfloor\lf(\frac{1}{p_{\Phi}^-}-1\r)
\frac{\ln b}{\ln(\lambda_-)}\r\rfloor+2,\infty\r).
\end{align*}
By \cite[Remarks 2.7(iii) and 4.21(iv)]{yhyy},
we conclude that $L^{\Phi}(\rn)$ satisfies all the assumptions
of Definition \ref{HXA} with $X:=L^{\Phi}(\rn)$,
$p_{-}\in(0,p_{\Phi}^-]$, $\tz_0\in(0,p_{\Phi}^-)$,
and $p_0\in(p_{\Phi}^+,\infty)$. Moreover,
choose a $q_0\in(p_{\Phi}^+,1]$. Then,
from \cite[Remark 5.3]{zyyw} and \cite[(25)]{hcy22},
we deduce that, for any non-negative measurable functions
$\{f_k\}_{k=1}^\infty$ and any $B\in\CB$,
$$\sum_{k=1}^{{\infty}}\|f_k\|_{[L^{\Phi}(\rn)]^{\frac{1}{q_0}}}
\le\left\|\sum_{k=1}^{{\infty}}f_k\right\|_
{[L^{\Phi}(\rn)]^{\frac{1}{q_0}}}$$
and
$$\|\one_B\|_{L^{\Phi}(\rn)}
\gs\min\lf\{|B|^{\frac{1}{p_{\Phi}^-}},|B|^{\frac{1}{p_{\Phi}^+}}\r\}
\ge\min\lf\{|B|^{\frac{1}{q_0}},|B|^{\frac{1}{\tz_0}}\r\}.$$
Therefore, all the assumptions of Theorems \ref{s2t1}, \ref{s3t1},
and \ref{s3t2} are satisfied with $X:=L^{\Phi}(\rn)$.
Applying Theorems \ref{s2t1}, \ref{s3t1},
and \ref{s3t2}, we obtain the following conclusion.

\begin{theorem}\label{thorlicz}
Let $\Phi$ be an Orlicz function with lower type $p_{\Phi}^-$
and upper type $p_{\Phi}^+$ satisfying $0<p_{\Phi}^-\le p_{\Phi}^+<1$.
Then Theorems \ref{s2t1}, \ref{s3t1}, and \ref{s3t2} still
hold true with $X$ replaced by $L^{\Phi}(\rn)$.
\end{theorem}

\begin{remark}\label{rerlicz}
We point out that Theorem \ref{thorlicz}
even when $A:=2I_{n\times n}$
is completely new.
\end{remark}

\subsection{Orlicz-Slice Spaces\label{s6-appl2}}

Recently, Zhang et al. \cite{zyyw} originally introduced the
Orlicz-slice space on $\rn$, which generalizes both the slice
space in \cite{am19} and the Wiener-amalgam space in \cite{dj17}.
They also introduced the Orlicz-slice (local) Hardy spaces
and developed a complete real-variable theory of these function spaces
in \cite{zyy21,zyyw}. Next, we recall the definition of
anisotropic Orlicz-slice spaces.

\begin{definition}
Let $\ell\in\zz$, $q\in(0,\infty)$, and $\Phi$ be an Orlicz function.
The {\it anisotropic Orlicz-slice space}
$(E_\Phi^q)_{\ell,A}\lf(\rn\r)$ is defined to be the
set of all the measurable functions $f$ on $\rn$ such that
\begin{align*}
\|f\|_{(E_\Phi^q)_{\ell,A}\lf(\rn\r)}
:=\lf\{\int_{\rn}\lf[\frac{\|f\one_{x+B_{\ell}}\|_{L^\Phi(\rn)}}
{\|\one_{x+B_{\ell}}\|_{L^\Phi(\rn)}}\r]^q\,dx\r\}^{\frac{1}{q}}
<\infty,
\end{align*}
where $B_{\ell}$ is the same as in \eqref{B_k}
with $k$ replaced by $\ell$.
\end{definition}

Let $\ell\in\zz$, $q\in(0,1)$, $\Phi$ be an Orlicz function
with positive lower type $p_\Phi^-$ and positive upper type
$p_\Phi^+$ satisfying $0<p_{\Phi}^-\le p_{\Phi}^+<1$, and
\begin{align*}
N\in\nn\cap\lf[\displaystyle\lf\lfloor\lf(\frac{1}{\min\{p_\Phi^-,q\}}-1\r)
\frac{\ln b}{\ln(\lambda_-)}\r\rfloor+2,\infty\r).
\end{align*}
Then, by \cite[Remark 8.14]{wyy22},
we conclude that $(E_\Phi^q)_{\ell,A}(\rn)$ satisfies all the assumptions
of Definition \ref{HXA} with $X:=(E_\Phi^q)_{\ell,A}(\rn)$,
$p_-\in(0,\min\{p_\Phi^-,q\}]$, $\tz_0\in(0,p_-)$,
and $p_0\in(\max\{p_\Phi^+,q\},\fz)$. Moreover,
choose a $q_0=1$. On the one hand,
from \cite[Lemma 5.4]{zyyw}, we infer that,
for any non-negative measurable functions $\{f_k\}_{k=1}^\infty$,
$$\sum_{k=1}^{{\infty}}\|f_k\|_{[(E_\Phi^q)_{\ell,A}(\rn)]^{\frac{1}{q_0}}}
\le\left\|\sum_{k=1}^{{\infty}}f_k\right\|_
{[(E_\Phi^q)_{\ell,A}(\rn)]^{\frac{1}{q_0}}}.$$
On the other hand, we have, for any $B\in\CB$,
\begin{equation}\label{eqorliczs}
\|\one_B\|_{(E_\Phi^q)_{\ell,A}\lf(\rn\r)}
\gs\min\lf\{|B|,|B|^{\frac{1}{\tz_0}}\r\}.
\end{equation}
Indeed, for any $B\in\CB$ with $|B|\geq |B_\ell|$,
\begin{align}\label{eqos1}
\|\one_B\|_{(E_\Phi^q)_{\ell,A}\lf(\rn\r)}&=
\lf\{\int_{\rn}\lf[\frac{\|\one_B\one_{x+B_{\ell}}\|_{L^\Phi(\rn)}}
{\|\one_{x+B_{\ell}}\|_{L^\Phi(\rn)}}\r]^q\,dx\r\}^{\frac{1}{q}}\\
&\gs \lf(\int_{B}1\,dx\r)^{\frac{1}{q}}
=|B|^{\frac{1}{q}}.\noz
\end{align}
On the other hand, for any $x_0\in\rn$,
$k\in\zz$ with $|B_k|\le |B_\ell|$, $x\in B(x_0,\lambda_-^\ell)$,
and $\eta\in(0,p_\Phi^-)$,
by \cite[Remark 4.21(iv)]{yhyy}, we conclude that $L^{\Phi}(\rn)$ satisfies
Assumption \ref{F-SVMI} with $X:=L^{\Phi}(\rn)$, $u:=1/\eta$, and $p:=\eta$.
Thus, we obtain
\begin{align}\label{6.16.1}
\|\one_{x_0+B_k}\|^{\eta}_{L^\Phi(\rn)}
\gs \|[\cm(\one_{x_0+B_k})]^{1/\eta}\|^{\eta}_{L^\Phi(\rn)}.
\end{align}
For any $y\in x+B_\ell$, we have
\begin{align}\label{6.16.2}
\cm(\one_{x_0+B_k})(y)
\geq\frac{1}{|B_\ell|}\int_{x+B_\ell}\one_{x_0+B_k}(z)\,dz
\gs\frac{|B_k|}{|B_\ell|}.
\end{align}
Combining \eqref{6.16.1} and \eqref{6.16.2}, we conclude that
\begin{equation*}
\|\one_{x_0+B_k}\|^{\eta}_{L^\Phi(\rn)}
\gs\lf\|\lf[\frac{|B_k|}{|B_\ell|}(\one_{x+B_\ell})\r]^{1/\eta}\r\|^{\eta}
_{L^\Phi(\rn)}
\gs|B_k|\,\|\one_{x+B_\ell}\|^{\eta}_{L^\Phi(\rn)}.
\end{equation*}
Therefore, we obtain
\begin{align}\label{eqos2}
\|\one_{x_0+B_k}\|_{(E_\Phi^q)_{\ell,A}\lf(\rn\r)}
&=\lf\{\int_{\rn}\lf[\frac{\|\one_{x_0+B_k}\one_{x+B_{\ell}}\|
_{L^\Phi(\rn)}}
{\|\one_{x+B_{\ell}}\|_{L^\Phi(\rn)}}\r]^q\,dx\r\}^{\frac{1}{q}}\\
&\gs \lf\{\int_{B(x_0,\lambda_-^\ell)}\lf[\frac{\|\one_{x_0+B_k}\|
_{L^\Phi(\rn)}}
{\|\one_{x+B_{\ell}}\|_{L^\Phi(\rn)}}\r]^q\,dx\r\}^{\frac{1}{q}}\noz\\
&\gs|B_k|^{1/\eta}\lf\{\int_{B(x_0,\lambda_-^\ell)}1\,dx\r\}
^{\frac{1}{q}}\sim|B_k|^{1/\eta}.\noz
\end{align}
By \eqref{eqos1} and \eqref{eqos2}, we find that, for any $B\in\CB$,
$$\|\one_B\|_{(E_\Phi^q)_{\ell,A}\lf(\rn\r)}\gs|B|^{1/q}
\gs |B|
\  \mbox{if}\  |B|\geq|B_\ell|$$
and
$$\|\one_B\|_{(E_\Phi^q)_{\ell,A}\lf(\rn\r)}\gs|B|^{1/p_\Phi^-}
\gs|B|^{1/\tz_0}\  \mbox{if}\  |B|\le|B_\ell|.$$
This finishes the proof of \eqref{eqorliczs}.

Therefore, all the assumptions of Theorems \ref{s2t1}, \ref{s3t1},
and \ref{s3t2} are satisfied with $X:=(E_\Phi^q)_{\ell,A}(\rn)$.
Applying Theorems \ref{s2t1}, \ref{s3t1}, and \ref{s3t2},
we obtain the following conclusion.

\begin{theorem}\label{thorliczslice}
Let $\ell\in\zz$, $q\in(0,1)$,
and $\Phi$ be an Orlicz function with lower type $p_{\Phi}^-$
and upper type $p_{\Phi}^+$ satisfying $0<p_{\Phi}^-\le p_{\Phi}^+<1$.
Then Theorems \ref{s2t1}, \ref{s3t1}, and \ref{s3t2} still
hold true with $X$ replaced by $(E_\Phi^q)_{\ell,A}(\rn)$.
\end{theorem}

\begin{remark}\label{rerliczslice}
We point out that Theorem \ref{thorliczslice}
even when $A:=2I_{n\times n}$
is completely new.
\end{remark}

\subsection{Local Generalized Herz Spaces\label{s6-appl6}}
In what follows, we always let $\mathbb{R}_+:=(0,\infty)$.
A nonnegative function $\omega$ on $\rr_+$ is said to be
\emph{almost increasing} (resp. \emph{almost decreasing})
on $\rr_+$ if there exists a constant $C\in[1,\fz)$ such that,
for any $s,t\in\rr_+$ satisfying $s\le t$ (resp. $s\ge t$),
$$\omega(s)\le C\omega(t)$$
(see, for instance, \cite{kmrs16,lyh22}).
Now, we recall the concept of the
function class $M(\mathbb{R}_{+})$ as follows
(see, for instance, \cite{kmrs16,HSamko}).

\begin{definition}
The \emph{function class} $M(\mathbb{R}_{+})$ is defined to be the set
of all the positive functions $\omega$ on $\mathbb{R}_{+}$ such that,
for any $0<\delta_1<\delta_2<\infty$,
$$0<\inf_{t\in(\delta_1,\delta_2)}\omega(t)\leq
\sup_{t\in(\delta_1,\delta_2)}\omega(t)<\infty$$
and there exist four constants $\alpha_{0},\beta_{0},
\alpha_{\infty},\beta_{\infty}\in\mathbb{R}$ such that
\begin{itemize}
\item[\rm (i)] for any $t\in(0,1]$,
$\omega(t)t^{-\alpha_{0}}$ is almost increasing
and $\omega(t) t^{-\beta_{0}}$ is almost decreasing;

\item[\rm (ii)] for any $t\in[1,\infty)$,
$\omega(t)t^{-\alpha_{\infty}}$ is almost increasing
and $\omega(t)t^{-\beta_{\infty}}$ is almost decreasing.
\end{itemize}
\end{definition}

Next, we introduce anisotropic local generalized Herz spaces
as follows.

\begin{definition}\label{5.25.x1}
Let $A$ be a dilation, $p,q\in(0,\infty]$, and $\omega\in M(\mathbb{R}_{+})$.
The \emph{anisotropic local generalized Herz space}
$\dot{\mathcal{K}}_{A,\omega}^{p,q}(\mathbb{R}^{n})$ is defined to
be the set of all $f\in L^p_{\loc}(\rn\setminus\{\bf{0}\})$ such that
$$\|f\|_{\dot{\mathcal{K}}_{A,\omega}^{p,q}
\left(\mathbb{R}^{n}\right)}:=\left\{\sum_{k \in
\mathbb{Z}}\left[\omega\left(b^{k}\right)\right]^{q}\left\|f
\mathbf{1}_{B_k\backslash B_{k-1}}\right\|_{L^{p}\left(\mathbb{R}^{n}\right)}
^{q}\right\}^{\frac{1}{q}}<\fz$$
with the usual modification made when $q=\fz$,
where $b$ is the same as in \eqref{2.14.x1} and,
for any $k\in\zz$, $B_k$ the same as in \eqref{B_k}.
\end{definition}

Now, we introduce the following Matuszewska--Orlicz indices
(see, for instance, \cite{mo60,mo65}).

\begin{definition}
Let $\omega$ be a positive
function on $\mathbb{R}_{+}$. Then the
\emph{Matuszewska--Orlicz indices} $m_{0}(\omega)$,
$M_{0}(\omega)$, $m_{\infty}(\omega)$, and
$M_{\infty}(\omega)$ of $\omega$ are defined,
respectively, by setting, for any $h\in(0,\infty)$,
$$m_{0}(\omega):=\sup_{t\in(0,1)}
\frac{\ln (\varlimsup\limits_{h\to0^{+}}
\frac{\omega(ht)}{\omega(h)})}{\ln t},\
M_{0}(\omega):=\inf_{t\in(0,1)}
\frac{\ln(\varliminf\limits_{h\to0^{+}}
\frac{\omega(h t)}{\omega(h)})}{\ln t},$$
$$m_{\infty}(\omega):=\sup_{t\in(1,\infty)}
\frac{\ln(\varliminf\limits_{h\to\infty}
\frac{\omega(ht)}{\omega(h)})}{\ln t},\ \mbox{and}\
M_{\infty}(\omega):=\inf_{t\in(1,\infty)}
\frac{\ln(\varlimsup\limits_{h\to\infty}
\frac{\omega(ht)}{\omega(h)})}{\ln t}.$$
\end{definition}

The following property about the function class $M(\mathbb{R}_{+})$
and the Matuszewska--Orlicz indices can be found in \cite[Lemma 1.1.6]{lyh22}
(see also \cite[(6.4), (6.5), and (6.14)]{samko13}). In what follows,
for any $\omega\in M(\mathbb{R}_{+})$ and $t\in \rr_+$, $\omega^s(t)
:=[\omega(t)]^s$.

\begin{lemma}\label{5.28.x1}
Let $\omega\in M(\mathbb{R}_{+})$. Then, for any given $t\in(0,\fz)$,
it holds true that $1/\omega,\omega^t\in M(\mathbb{R}_{+})$ and
\begin{itemize}
\item[\rm (i)]
$m_0(1/\omega)=-M_0(\omega)$ and $M_0(1/\omega)=-m_0(\omega)$;

\item[\rm (ii)]
$m_{\fz}(1/\omega)=-M_{\fz}(\omega)$ and $M_{\fz}(1/\omega)=-m_{\fz}(\omega)$;

\item[\rm (iii)]
$m_0(\omega^t)=tm_0(\omega)$ and $M_0(\omega^t)=tM_0(\omega)$;

\item[\rm (iv)]
$m_{\fz}(\omega^t)=tm_{\fz}(\omega)$ and $M_{\fz}(\omega^t)=tM_{\fz}(\omega)$.
\end{itemize}
\end{lemma}

\begin{remark}
Let $\omega\in M(\mathbb{R}_{+})$. Then,
by \cite[Remark 1.1.5(ii)]{lyh22}, we conclude that
\begin{align*}
-\fz<m_{0}(\omega)\le M_{0}(\omega)<\fz
\end{align*}
and
\begin{align*}
-\fz<m_{\fz}(\omega)\le M_{\fz}(\omega)<\fz.
\end{align*}
\end{remark}

The following theorem shows that anisotropic
local generalized Herz spaces
$\dot{\mathcal{K}}_{A,\omega}^{p,q}\left(\mathbb{R}^{n}\right)$
are ball quasi-Banach function spaces under some additional
assumptions on the exponent $\omega\in M(\mathbb{R}_{+})$;
see \cite[Theorem 1.2.42]{lyh22} for the standard Euclidean space case.

\begin{theorem}\label{5.27.x1}
Let $A$ be a dilation, $p,q\in(0,\infty]$,
and $\omega\in M(\mathbb{R}_{+})$ with
$m_{0}(\omega)\in(-\frac{1}{p},\fz)$. Then
the anisotropic local generalized Herz space
$\dot{\mathcal{K}}_{A,\omega}^{p,q}\left(\mathbb{R}^{n}\right)$
is a ball quasi-Banach function space.
\end{theorem}

\begin{proof}
Indeed, by an argument similar to that used in the
proof of  \cite[Theorem 1.2.38]{lyh22}, we find that
$\dot{\mathcal{K}}_{A,\omega}^{p,q}\left(\mathbb{R}^{n}\right)$
is a quasi-Banach space satisfying (i), (ii), and (iii)
of Definition \ref{BQBFS}. Next, we prove that
$\dot{\mathcal{K}}_{A,\omega}^{p,q}\left(\mathbb{R}^{n}\right)$
satisfies Definition \ref{BQBFS}(iv). To this end,
let $x_0\in\rn$, $k_0\in\zz$,
and $B_{k_0}$ be the same as in \eqref{B_k}. Then
\begin{align}\label{5.27.x2}
\lf\|\one_{x_0+B_{k_0}}\r\|
_{\dot{\mathcal{K}}_{A,\omega}^{p,q}\left(\mathbb{R}^{n}\right)}
&=\left\{\sum_{k\in\mathbb{Z}}
\left[\omega\left(b^{k}\right)\right]^{q}
\left\|\one_{x_0+B_{k_0}}\mathbf{1}_{B_k\backslash B_{k-1}}\right\|_{L^{p}
\left(\mathbb{R}^{n}\right)}^{q}\right\}^{\frac{1}{q}}\\
&\ls\left\{\sum_{k\in\mathbb{Z}\backslash\nn}
\left[\omega\left(b^{k}\right)\right]^{q}
\left\|\one_{x_0+B_{k_0}}\mathbf{1}_{B_k\backslash B_{k-1}}\right\|_{L^{p}
\left(\mathbb{R}^{n}\right)}^{q}\right\}^{\frac{1}{q}}
+\left\{\sum_{k\in\nn}\cdots\right\}^{\frac{1}{q}}\noz\\
&=:{\rm I}+{\rm II}.\noz
\end{align}
We first deal with ${\rm I}$. From \cite[Lemma 1.1.12]{lyh22},
we deduce that, for any $k\in\zz\backslash\nn$,
\begin{align*}
\omega\left(b^{k}\right)\ls b^{k[m_0(\omega)-\varepsilon]},
\end{align*}
where $\varepsilon\in(0,m_0(\omega)+\frac1p)$
is a fixed positive constant. This, combined with \eqref{5.27.x2},
further implies that
\begin{align}\label{5.27.x3}
{\rm I}&\ls\left\{\sum_{k\in\mathbb{Z}\backslash\nn}
b^{kq[m_0(\omega)-\varepsilon]}
\left\|\mathbf{1}_{B_k\backslash B_{k-1}}\right\|_{L^{p}
\left(\mathbb{R}^{n}\right)}^{q}\right\}^{\frac{1}{q}}\\
&\ls\left\{\sum_{k\in\mathbb{Z}\backslash\nn}
b^{kq[m_0(\omega)+\frac1p-\varepsilon]}\right\}^{\frac{1}{q}}
<\fz.\noz
\end{align}
Next, we deal with ${\rm II}$. To this end, we first claim that,
for any $k\in\nn\cap(\tau+|k_0|+|\lfloor\log_b\rho(x_0)\rfloor|+3,\fz)$
with $\tau$ in \eqref{tau},
\begin{align}\label{5.27.y1}
\lf(x_0+B_{k_0}\r)\cap \lf(B_{k}\backslash B_{k-1}\r)=\emptyset.
\end{align}
Indeed, by the inequality in line 26 of \cite[p.\,7]{Bownik}, we conclude that,
for any $y\in x_0+B_{k_0}$,
\begin{align*}
\rho(y)&\le b^{\tau}\lf[\rho(x_0-y)+\rho(x_0)\r]
<b^{\tau}\lf[b^{k_0-1}+b^{\lfloor\log_b\rho(x_0)\rfloor+1}\r]\\
&\le b^{\tau+|k_0|+|\lfloor\log_b\rho(x_0)\rfloor|+1}<b^{k-2},
\end{align*}
which further implies $y\in B_{k-1}$ and hence
$x_0+B_{k_0}\subset B_{k-1}$. Thus, \eqref{5.27.y1} holds true.
From \eqref{5.27.x2} and \eqref{5.27.y1},
we deduce that
\begin{align*}
{\rm II}
&=\left\{\sum_{k\in\nn\cap[1,\tau+|k_0|+|\lfloor\log_b\rho(x_0)\rfloor|+3]}
\left[\omega\left(b^{k}\right)\right]^{q}
\left\|\one_{x_0+B_{k_0}}\mathbf{1}_{B_k\backslash B_{k-1}}\right\|_{L^{p}
\left(\mathbb{R}^{n}\right)}^{q}\right\}^{\frac{1}{q}}<\fz.
\end{align*}
This, together with \eqref{5.27.x2} and \eqref{5.27.x3}, then
finishes the proof of Theorem \ref{5.27.x1}.
\end{proof}

The following lemma gives the Fefferman--Stein vector-valued
inequality on anisotropic local generalized Herz spaces;
see \cite[Theorem 1.6.1]{lyh22} for the standard Euclidean space case.

\begin{lemma}\label{vmbhl}
Let $A$ be a dilation, $p,r\in(1,\infty]$, $q\in(0,\fz]$,
and $\omega\in M(\mathbb{R}_{+})$ satisfy
$$-\frac{1}{p}<m_{0}(\omega)\le M_{0}(\omega)<\frac{1}{p'}$$
and
$$-\frac{1}{p}<m_{\fz}(\omega)\le M_{\fz}(\omega)<\frac{1}{p'},$$
where $\frac{1}{p}+\frac{1}{p'}=1$.
Then there exists a positive constant $C$ such that,
for any $\{f_j\}_{j\in\mathbb{N}}\subset L_{{\rm loc}}^{1}(\rn)$,
\begin{equation*}
\left\|\left\{\sum_{j\in\mathbb{N}}
\left[\cm(f_{j})\right]^{r}\right\}^{\frac{1}{r}}
\right\|_{\dot{\mathcal{K}}_{A,\omega}^{p,q}(\mathbb{R}^{n})}
\leq C\left\|\left(\sum_{j\in\mathbb{N}}|f_{j}|^{r}\right)^{\frac{1}{r}}
\right\|_{\dot{\mathcal{K}}_{A,\omega}^{p,q}(\mathbb{R}^{n})}.
\end{equation*}
\end{lemma}

To prove this lemma, we need more preparations.
Recall that an operator $T$ defined on $\mathscr M(\rn)$
is called a \emph{sublinear operator} if,
for any $f,g\in\mathscr M(\rn)$ and $\lz\in\cc$,
$$|T(f+g)|\le|T(f)|+|T(g)|$$
and
$$|T(\lz f)|=|\lz||T(f)|.$$
Moreover, for any normed linear space $X$ and any operator $T$ on $X$,
the \emph{operator norm $\|T\|_{X\rightarrow X}$}
of $T$ is defined by setting
$$\|T\|_{X\rightarrow X}:=\sup_{\{x\in X:\ \|x\|_X=1\}}\|T(x)\|_X.$$
The following lemma is a boundedness criterion of sublinear
operators on anisotropic local generalized Herz spaces;
see \cite[Theorem 1.5.1]{lyh22} for the standard Euclidean space case.
\begin{lemma}\label{vmbhl-lem}
Let $A$ be a dilation, $p\in(1,\infty]$, $q\in(0,\fz]$,
and $\omega\in M(\mathbb{R}_{+})$ satisfy
$$-\frac{1}{p}<m_{0}(\omega)\le M_{0}(\omega)<\frac{1}{p'}$$
and
$$-\frac{1}{p}<m_{\fz}(\omega)\le M_{\fz}(\omega)<\frac{1}{p'},$$
where $\frac{1}{p}+\frac{1}{p'}=1$. Assume that $T$ is a sublinear
operator satisfying that $T$ is bounded on $L^p(\rn)$ and that
there exists a positive constant $C_0$ such that,
for any $f\in\dot{\mathcal{K}}_{A,\omega}^{p,q}(\mathbb{R}^{n})$
and $x\notin\overline{\supp(f)}:=\overline{\{x\in\rn:\ f(x)\neq0\}}$,
\begin{align}\label{5.22.x1}
|T(f)(x)|&\le C_0\int_{\rn}\frac{|f(y)|}{\rho(x-y)}\,dy,
\end{align}
where $\rho$ is the same as in Definition \ref{def-shqn}.
Then there exists a positive constant $C$, independent of $T$,
such that, for any $f\in\dot{\mathcal{K}}_{A,\omega}^{p,q}(\mathbb{R}^{n})$,
\begin{align}\label{5.25.z1}
\|T(f)\|_{\dot{\mathcal{K}}_{A,\omega}^{p,q}(\mathbb{R}^{n})}
\le C\lf[C_0+\|T\|_{L^p(\rn)\rightarrow L^p(\rn)}\r]
\|f\|_{\dot{\mathcal{K}}_{A,\omega}^{p,q}(\mathbb{R}^{n})}.
\end{align}
\end{lemma}

\begin{proof}
Let $f\in\dot{\mathcal{K}}_{A,\omega}^{p,q}(\mathbb{R}^{n})$.
For any given $k\in\zz$, let
$$f_{k,1}:=f\one_{B_{k-\tau-1}},\
f_{k,2}:=f\one_{B_{k+\tau+2}\backslash B_{k-\tau-1}},\
\mbox{and}\  f_{k,3}:=f\one_{[B_{k+\tau+2}]^{\complement}},$$
where $\tau$ is the same as in \eqref{tau}.
Obviously, for any $k\in\zz$, it holds true that
$$f=f_{k,1}+f_{k,2}+f_{k,3}.$$
From this, Definition \ref{5.25.x1}, the sublinearity of $T$,
and the Minkowski inequality, we deduce that
\begin{align}\label{5.22.x3}
\|T(f)\|_{\dot{\mathcal{K}}_{A,\omega}^{p,q}(\mathbb{R}^{n})}
&=\left\{\sum_{k\in\mathbb{Z}}\left[\omega\left(b^{k}\right)\right]^{q}
\left\|T(f)\mathbf{1}_{B_k\backslash B_{k-1}}\right\|_{L^{p}
\left(\mathbb{R}^{n}\right)}^{q}\right\}^{\frac{1}{q}}\\
&\ls\left\{\sum_{k\in\mathbb{Z}}\left[\omega\left(b^{k}\right)\right]^{q}
\left\|T(f_{k,1})\mathbf{1}_{B_k\backslash B_{k-1}}\right\|_{L^{p}
\left(\mathbb{R}^{n}\right)}^{q}\right\}^{\frac{1}{q}}\noz\\
&\quad+\left\{\sum_{k\in\mathbb{Z}}\left[\omega\left(b^{k}\right)\right]^{q}
\left\|T(f_{k,2})\mathbf{1}_{B_k\backslash B_{k-1}}\right\|_{L^{p}
\left(\mathbb{R}^{n}\right)}^{q}\right\}^{\frac{1}{q}}\noz\\
&\quad+\left\{\sum_{k\in\mathbb{Z}}\left[\omega\left(b^{k}\right)\right]^{q}
\left\|T(f_{k,3})\mathbf{1}_{B_k\backslash B_{k-1}}\right\|_{L^{p}
\left(\mathbb{R}^{n}\right)}^{q}\right\}^{\frac{1}{q}}\noz\\
&=:{\rm I}_1+{\rm I}_2+{\rm I}_3,\noz
\end{align}
where the implicit positive constant is independent of both $T$ and $f$.
Next, we deal with ${\rm I}_1$, ${\rm I}_2$, and ${\rm I}_3$ successively.
To this end, let
\begin{align}\label{5.22.z1}
\varepsilon\in\lf(0,\min\lf\{\min\{m_{0}(\omega),m_{\fz}(\omega)\}+\frac1p,
\frac{1}{p'}-\max\{M_{0}(\omega),M_{\fz}(\omega)\}\r\}\r)
\end{align}
be a fixed positive constant. Then, by \cite[Lemma 1.5.2]{lyh22},
we find that, for any $0<s<t<\fz$,
\begin{align}\label{5.23.x1}
\frac{\omega(s)}{\omega(t)}
&\ls\lf(\frac{s}{t}\r)^{\min\{m_{0}(\omega),m_{\fz}(\omega)\}-\varepsilon}
\end{align}
and
\begin{align}\label{5.22.y1}
\frac{\omega(t)}{\omega(s)}
&\ls\lf(\frac{t}{s}\r)^{\max\{M_{0}(\omega),M_{\fz}(\omega)\}+\varepsilon},
\end{align}
where the implicit positive constant is independent of both $T$ and $f$.

We first deal with ${\rm I}_1$.
From the inequality in line 26 of \cite[p.\,7]{Bownik},
we infer that, for any $k,i\in\zz$ satisfying $i\in(-\fz,k-\tau-1]$
and for any
$x,y\in\rn$ satisfying $x\in B_k\backslash B_{k-1}$
and $y\in B_i\backslash B_{i-1}$,
\begin{align*}
\rho(x-y)&\ge b^{-\tau}\rho(x)-\rho(y)=b^{k-\tau-1}-b^{i-1}\\
&\ge b^{k-\tau-1}-b^{k-\tau-2}=(b^{-\tau-1}-b^{-\tau-2})b^{k}.
\end{align*}
By this, \eqref{5.22.x1}, and the H\"older inequality,
we conclude that, for any $k\in\zz$ and $x\in B_k\backslash B_{k-1}$,
\begin{align}\label{5.22.x2}
\lf|T(f_{k,1})(x)\r|&\le C_0\int_{\rn}\frac{|f_{k,1}(y)|}{\rho(x-y)}\,dy
=C_0\sum_{i=-\fz}^{k-\tau-1}\int_{B_i\backslash B_{i-1}}\frac{|f(y)|}{\rho(x-y)}\,dy\\
&\ls C_0\sum_{i=-\fz}^{k-\tau-1}b^{-k+\frac{i}{p'}}
\lf\|f\one_{B_i\backslash B_{i-1}}\r\|_{L^p(\rn)},\noz
\end{align}
where the implicit positive constant is independent of both $T$ and $f$.
Moreover, from \eqref{5.22.y1}, it follows that,
for any $k,i\in\zz$ satisfying $i\in(-\fz,k-\tau-1]$,
\begin{align*}
\frac{\omega(b^k)}{\omega(b^i)}
&\ls b^{(k-i)[\max\{M_{0}(\omega),M_{\fz}(\omega)\}+\varepsilon]},
\end{align*}
which, combined with \eqref{5.22.x2}, further implies that,
for any $k\in\zz$,
\begin{align}\label{5.22.y2}
&\omega\lf(b^k\r)\lf\|T(f_{k,1})\one_{B_k\backslash B_{k-1}}\r\|_{L^p(\rn)}\\
&\quad\ls C_0\omega\lf(b^k\r)\sum_{i=-\fz}^{k-\tau-1}b^{-k+\frac{i}{p'}}
\lf\|f\one_{B_i\backslash B_{i-1}}\r\|_{L^p(\rn)}
\lf\|\one_{B_k\backslash B_{k-1}}\r\|_{L^p(\rn)}\noz\\
&\quad\ls C_0\sum_{i=-\fz}^{k-\tau-1}b^{\frac{1}{p'}(i-k)}\frac{\omega(b^k)}{\omega(b^i)}
\omega\lf(b^i\r)\lf\|f\one_{B_i\backslash B_{i-1}}\r\|_{L^p(\rn)}\noz\\
&\quad\ls C_0\sum_{i=-\fz}^{k-\tau-1}b^{(k-i)
[\max\{M_{0}(\omega),M_{\fz}(\omega)\}+\varepsilon-\frac{1}{p'}]}
\omega\lf(b^i\r)\lf\|f\one_{B_i\backslash B_{i-1}}\r\|_{L^p(\rn)},\noz
\end{align}
where the implicit positive constants are independent of both $T$ and $f$.
Next, we show the desired estimate of ${\rm I}_1$
by considering the following two cases on $q$.

Case 1) $q\in(0,1]$. In this case, by \eqref{5.22.x3}, \eqref{5.22.y2},
Lemma \ref{basicine}, and \eqref{5.22.z1}, we conclude that
\begin{align}\label{5.22.y3}
{\rm I}_1&\ls C_0\left\{\sum_{k\in\mathbb{Z}}\sum_{i=-\fz}^{k-\tau-1}
b^{q(k-i)[\max\{M_{0}(\omega),M_{\fz}(\omega)\}+\varepsilon-\frac{1}{p'}]}
\lf[\omega\lf(b^i\r)\r]^q\lf\|f\one_{B_i\backslash B_{i-1}}\r\|_{L^p(\rn)}^q
\right\}^{\frac{1}{q}}\\
&=C_0\left\{\sum_{i\in\mathbb{Z}}\sum_{k=i+\tau+1}^{\fz}
b^{q(k-i)[\max\{M_{0}(\omega),M_{\fz}(\omega)\}+\varepsilon-\frac{1}{p'}]}
\lf[\omega\lf(b^i\r)\r]^q\lf\|f\one_{B_i\backslash B_{i-1}}\r\|_{L^p(\rn)}^q
\right\}^{\frac{1}{q}}\noz\\
&\sim C_0\left\{\sum_{i\in\mathbb{Z}}
\lf[\omega\lf(b^i\r)\r]^q\lf\|f\one_{B_i\backslash B_{i-1}}\r\|_{L^p(\rn)}^q
\right\}^{\frac{1}{q}}\noz\\
&\le\lf[C_0+\|T\|_{L^p(\rn)\rightarrow L^p(\rn)}\r]
\|f\|_{\dot{\mathcal{K}}_{A,\omega}^{p,q}(\mathbb{R}^{n})},\noz
\end{align}
where the implicit positive constants are independent of both $T$ and $f$.

Case 2) $q\in(1,\fz]$. In this case, from
\eqref{5.22.y2}, the H\"older inequality,
and \eqref{5.22.z1}, we deduce that, for any $k\in\zz$,
\begin{align}\label{5.25.y2}
&\omega\lf(b^k\r)\lf\|T(f_{k,1})\one_{B_k\backslash B_{k-1}}\r\|_{L^p(\rn)}\\
&\quad\ls C_0\lf\{\sum_{i=-\fz}^{k-\tau-1}b^{\frac{(k-i)q'}{2}
[\max\{M_{0}(\omega),M_{\fz}(\omega)\}+\varepsilon-\frac{1}{p'}]}\r\}
^{\frac{1}{q'}}\noz\\
&\qquad\times\lf\{\sum_{i=-\fz}^{k-\tau-1}b^{\frac{(k-i)q}{2}
[\max\{M_{0}(\omega),M_{\fz}(\omega)\}+\varepsilon-\frac{1}{p'}]}
\lf[\omega\lf(b^i\r)\r]^q\lf\|f\one_{B_i\backslash B_{i-1}}\r\|_{L^p(\rn)}^q\r\}^{\frac1q}\noz\\
&\quad\sim C_0\lf\{\sum_{i=-\fz}^{k-\tau-1}b^{\frac{(k-i)q}{2}
[\max\{M_{0}(\omega),M_{\fz}(\omega)\}+\varepsilon-\frac{1}{p'}]}
\lf[\omega\lf(b^i\r)\r]^q\lf\|f\one_{B_i\backslash B_{i-1}}\r\|_{L^p(\rn)}^q\r\}^{\frac1q},\noz
\end{align}
which, together with \eqref{5.22.x3} and \eqref{5.22.z1},
further implies that
\begin{align}\label{5.25.y3}
{\rm I}_1&\ls C_0\left\{\sum_{k\in\mathbb{Z}}\sum_{i=-\fz}^{k-\tau-1}
b^{\frac{(k-i)q}{2}[\max\{M_{0}(\omega),M_{\fz}(\omega)\}+\varepsilon-\frac{1}{p'}]}
\lf[\omega\lf(b^i\r)\r]^q\lf\|f\one_{B_i\backslash B_{i-1}}\r\|_{L^p(\rn)}^q
\right\}^{\frac{1}{q}}\\
&= C_0\left\{\sum_{i\in\mathbb{Z}}\sum_{k=i+\tau+1}^{\fz}
b^{\frac{(k-i)q}{2}[\max\{M_{0}(\omega),M_{\fz}(\omega)\}+\varepsilon-\frac{1}{p'}]}
\lf[\omega\lf(b^i\r)\r]^q\lf\|f\one_{B_i\backslash B_{i-1}}\r\|_{L^p(\rn)}^q
\right\}^{\frac{1}{q}}\noz\\
&\sim C_0\left\{\sum_{i\in\mathbb{Z}}
\lf[\omega\lf(b^i\r)\r]^q\lf\|f\one_{B_i\backslash B_{i-1}}\r\|_{L^p(\rn)}^q
\right\}^{\frac{1}{q}}\noz\\
&\le\lf[C_0+\|T\|_{L^p(\rn)\rightarrow L^p(\rn)}\r]
\|f\|_{\dot{\mathcal{K}}_{A,\omega}^{p,q}(\mathbb{R}^{n})},\noz
\end{align}
where the implicit positive constants are independent of both $T$ and $f$.

Second, we deal with ${\rm I}_2$.
By the boundedness of $T$ on $L^p(\rn)$ and the Minkowski inequality,
we find that, for any $k\in\zz$,
\begin{align}\label{5.22.y4}
\lf\|T(f_{k,2})\one_{B_k\backslash B_{k-1}}\r\|_{L^p(\rn)}
&\le\lf\|T(f_{k,2})\r\|_{L^p(\rn)}
\le\|T\|_{L^p(\rn)\rightarrow L^p(\rn)}\lf\|f_{k,2}\r\|_{L^p(\rn)}\\
&\le\|T\|_{L^p(\rn)\rightarrow L^p(\rn)}
\sum_{i=-\tau}^{\tau+2}\lf\|f\one_{B_{k+i}\backslash B_{k+i-1}}\r\|_{L^p(\rn)}.\noz
\end{align}
Moreover, from \cite[Lemma 1.1.3]{lyh22}, it follows that,
for any $k\in\zz$ and $i\in[-\tau,\tau+2]\cap\zz$,
\begin{align*}
\omega\lf(b^k\r)&\sim\omega\lf(b^{k+i}\r),
\end{align*}
which, combined with \eqref{5.22.x3} and \eqref{5.22.y4}, further implies that
\begin{align}\label{5.22.z2}
{\rm I}_2&\le\|T\|_{L^p(\rn)\rightarrow L^p(\rn)}
\left\{\sum_{k\in\mathbb{Z}}\lf[\omega\lf(b^k\r)\r]^q
\sum_{i=-\tau}^{\tau+2}\lf\|f\one_{B_{k+i}\backslash B_{k+i-1}}\r\|_{L^p(\rn)}^q
\right\}^{\frac{1}{q}}\\
&\ls\|T\|_{L^p(\rn)\rightarrow L^p(\rn)}\sum_{i=-\tau}^{\tau+2}
\left\{\sum_{k\in\mathbb{Z}}\lf[\omega\lf(b^{k+i}\r)\r]^q
\lf\|f\one_{B_{k+i}\backslash B_{k+i-1}}\r\|_{L^p(\rn)}^q
\right\}^{\frac{1}{q}}\noz\\
&\sim\|T\|_{L^p(\rn)\rightarrow L^p(\rn)}
\left\{\sum_{k\in\mathbb{Z}}\lf[\omega\lf(b^{k}\r)\r]^q
\lf\|f\one_{B_{k}\backslash B_{k-1}}\r\|_{L^p(\rn)}^q\right\}^{\frac{1}{q}}\noz\\
&\le\lf[C_0+\|T\|_{L^p(\rn)\rightarrow L^p(\rn)}\r]
\|f\|_{\dot{\mathcal{K}}_{A,\omega}^{p,q}(\mathbb{R}^{n})}.\noz
\end{align}

Finally, we deal with ${\rm I}_3$. By
the inequality in line 26 of \cite[p.\,7]{Bownik}, we conclude that,
for any $k,i\in\zz$ satisfying $i\in[k+\tau+3,\fz)$
and for any
$x,y\in\rn$ satisfying $x\in B_k\backslash B_{k-1}$
and $y\in B_i\backslash B_{i-1}$,
\begin{align*}
\rho(x-y)&\ge b^{-\tau}\rho(y)-\rho(x)=b^{i-\tau-1}-b^{k-1}\\
&\ge b^{i-\tau-1}-b^{i-\tau-4}=(b^{-\tau-1}-b^{-\tau-4})b^{i}.
\end{align*}
From this, \eqref{5.22.x1}, and the H\"older inequality,
we infer that, for any $k\in\zz$ and $x\in B_k\backslash B_{k-1}$,
\begin{align}\label{5.23.x2}
\lf|T(f_{k,3})(x)\r|&\le C_0\int_{\rn}\frac{|f_{k,3}(y)|}{\rho(x-y)}\,dy
=C_0\sum_{i=k+\tau+3}^{\fz}\int_{B_i\backslash B_{i-1}}\frac{|f(y)|}{\rho(x-y)}\,dy\\
&\ls C_0\sum_{i=k+\tau+3}^{\fz}b^{-\frac{i}{p}}
\lf\|f\one_{B_i\backslash B_{i-1}}\r\|_{L^p(\rn)},\noz
\end{align}
where the implicit positive constant is independent of both $T$ and $f$.
Moreover, by \eqref{5.23.x1}, we conclude that,
for any $k,i\in\zz$ satisfying $i\in[k+\tau+3,\fz)$,
\begin{align*}
\frac{\omega(b^k)}{\omega(b^i)}
&\ls b^{(k-i)[\min\{m_{0}(\omega),m_{\fz}(\omega)\}-\varepsilon]},
\end{align*}
which, combined with \eqref{5.23.x2}, further implies that,
for any $k\in\zz$,
\begin{align}\label{5.25.z2}
&\omega\lf(b^k\r)\lf\|T(f_{k,3})\one_{B_k\backslash B_{k-1}}\r\|_{L^p(\rn)}\\
&\quad\ls C_0\omega\lf(b^k\r)\sum_{i=k+\tau+3}^{\fz}b^{-\frac{i}{p}}
\lf\|f\one_{B_i\backslash B_{i-1}}\r\|_{L^p(\rn)}
\lf\|\one_{B_k\backslash B_{k-1}}\r\|_{L^p(\rn)}\noz\\
&\quad\ls C_0\sum_{i=k+\tau+3}^{\fz}b^{\frac{1}{p}(k-i)}\frac{\omega(b^k)}{\omega(b^i)}
\omega\lf(b^i\r)\lf\|f\one_{B_i\backslash B_{i-1}}\r\|_{L^p(\rn)}\noz\\
&\quad\ls C_0\sum_{i=k+\tau+3}^{\fz}b^{(k-i)
[\min\{m_{0}(\omega),m_{\fz}(\omega)\}-\varepsilon+\frac{1}{p}]}
\omega\lf(b^i\r)\lf\|f\one_{B_i\backslash B_{i-1}}\r\|_{L^p(\rn)},\noz
\end{align}
where the implicit positive constants are independent of both $T$ and $f$.
From \eqref{5.25.z2}, \eqref{5.22.z1}, and an argument similar to that
used in the estimations of \eqref{5.22.y3} and \eqref{5.25.y3},
we deduce that
\begin{align*}
{\rm I}_3&\ls\lf[C_0+\|T\|_{L^p(\rn)\rightarrow L^p(\rn)}\r]
\|f\|_{\dot{\mathcal{K}}_{A,\omega}^{p,q}(\mathbb{R}^{n})},
\end{align*}
where the implicit positive constant is independent of both $T$ and $f$.
This, together with \eqref{5.22.x3}, \eqref{5.22.y3}, \eqref{5.25.y3},
and \eqref{5.22.z2}, further implies \eqref{5.25.z1} and
hence finishes the proof of Lemma \ref{vmbhl-lem}.
\end{proof}

The following conclusion is a simple corollary of both
\eqref{5.23.x3} and Lemma \ref{vmbhl-lem}.

\begin{corollary}\label{5.23.y2}
Let $A$, $p$, $q$, and $\omega$ be the same as in Lemma \ref{vmbhl-lem}
and $\cm$ the Hardy--Littlewood maximal operator in \eqref{5.23.x3}.
Then there exists a positive constant $C$ such that, for any $f\in\dot{\mathcal{K}}_{A,\omega}^{p,q}(\mathbb{R}^{n})$,
\begin{align*}
\lf\|\cm(f)\r\|_{\dot{\mathcal{K}}_{A,\omega}^{p,q}(\mathbb{R}^{n})}
&\le C\|f\|_{\dot{\mathcal{K}}_{A,\omega}^{p,q}(\mathbb{R}^{n})}.
\end{align*}
\end{corollary}

Now, we show Lemma \ref{vmbhl}.

\begin{proof}[Proof of Lemma \ref{vmbhl}]
Let $\{f_j\}_{j\in\nn}\st L_{\loc}^1(\rn)$ be any given sequence and,
for any $r\in(1,\infty]$, $g\in\mathscr M(\rn)$, and $x\in\rn$, define
$$A(g)(x):=\lf\{\sum_{j\in\nn}
\lf[\cm(g\xi_j)(x)\r]^r\r\}^{\frac{1}{r}},$$
where, for any $j\in\nn$ and $y\in\rn$,
$$\xi_j(y):=\frac{f_j(y)}
{[\sum_{j\in\nn}|f_j(y)|^r]^{1/r}}\
{\rm if}\ \lf[\sum_{j\in\nn}\lf|f_j(y)\r|^r\r]^{1/r}\neq0$$
and $\xi_j(y):=0$ otherwise.
Obviously, for any $\lz\in\cc$ and $g\in\mathscr M(\rn)$,
$$A(\lz g)=|\lz|A(g).$$
Moreover, by the Minkowski inequality, we find that,
for any $g_1,g_2\in\mathscr M(\rn)$ and $x\in\rn$,
\begin{align*}
A(g_1+g_2)(x)
&=\lf\{\sum_{j\in\nn}\lf[
\cm\lf([g_1+g_2]\xi_j\r)(x)\r]^r\r\}^{\frac{1}{r}}\\
&\le\lf\{\sum_{j\in\nn}\lf[\cm(g_1\xi_j)(x)+
\cm(g_2\xi_j)(x)\r]^r\r\}^{\frac{1}{r}}\\
&\le\lf\{\sum_{j\in\nn}\lf[\cm(g_1\xi_j)(x)\r]^r\r\}^{\frac{1}{r}}
+\lf\{\sum_{j\in\nn}\lf[\cm(g_2\xi_j)(x)\r]^r\r\}^{\frac{1}{r}}\\
&=A(g_1)(x)+A(g_2)(x).
\end{align*}
Thus, $A$ is sublinear. Next, we prove that $A$ satisfies \eqref{5.22.x1}.
Indeed, notice that, for any $y\in\rn$,
\begin{align}\label{5.23.y1}
\lf[\sum_{j\in\nn}\lf|\xi_j(y)\r|^r\r]^{1/r}\le 1,
\end{align}
which, combined with \eqref{5.23.x3} and the Minkowski integral inequality,
further implies that, for any $x\notin\overline{\supp(g)}$,
\begin{align}\label{5.23.x4}
\lf|A(g)(x)\r|&=\lf\{\sum_{j\in\nn}\lf[\cm(g\xi_j)(x)\r]^r\r\}^{\frac{1}{r}}
\ls\lf\{\sum_{j\in\nn}\lf[\int_{\rn}
\frac{|g(y)\xi_j(y)|}{\rho(x-y)}\,dy\r]^r\r\}^{\frac{1}{r}}\\
&\le\int_{\rn}\lf\{\sum_{j\in\nn}
\lf[\frac{|g(y)\xi_j(y)|}{\rho(x-y)}\r]^r\r\}^{\frac{1}{r}}\,dy
\le\int_{\rn}\frac{|g(y)|}{\rho(x-y)}\,dy,\noz
\end{align}
where the implicit positive constant is independent of $\{f_j\}_{j\in\nn}$.
Thus, $A$ satisfies \eqref{5.22.x1}. In addition, from the Fefferman--Stein
vector-valued inequality of $\cm$ on the Lebesgue space $L^r(\rn)$
(see, for instance, \cite[p.\,104, Theorem 6.2]{Bownik})
and \eqref{5.23.y1}, we infer that, for any $g\in\mathscr M(\rn)$,
\begin{align}\label{5.23.z1}
\lf\|A(g)\r\|_{L^{p}(\rn)}
&=\lf\|\lf\{\sum_{j\in\nn}\lf[\cm(g\xi_j)\r]^r\r\}^{\frac{1}{r}}\r\|_{L^{p}(\rn)}\\
&\ls\lf\|\lf(\sum_{j\in\nn}|g\xi_j|^r\r)^{\frac{1}{r}}\r\|_{L^{p}(\rn)}
\le\|g\|_{L^{p}(\rn)},\noz
\end{align}
where the implicit positive constant is independent of $\{f_j\}_{j\in\nn}$.
Thus, $A$ satisfies all the assumptions of Lemma \ref{vmbhl-lem}.
This, together with the fact that the implicit positive constants
in both \eqref{5.23.x4} and \eqref{5.23.z1} are independent of $\{f_j\}_{j\in\nn}$,
further implies that $A$ is bounded on $\dot{\mathcal{K}}_{A,\omega}^{p,q}(\mathbb{R}^{n})$.
Letting $g:=(\sum_{j\in\nn}|f_j|^r)^{1/r}$, by Lemma \ref{vmbhl-lem}, we conclude that
\begin{align*}
\lf\|\lf\{\sum_{j\in\nn}\lf[\cm(f_j)\r]^r\r\}^{\frac{1}{r}}\r\|
_{\dot{\mathcal{K}}_{A,\omega}^{p,q}(\mathbb{R}^{n})}
&=\lf\|A(g)\r\|_{\dot{\mathcal{K}}_{A,\omega}^{p,q}(\mathbb{R}^{n})}\\
&\ls\|g\|_{\dot{\mathcal{K}}_{A,\omega}^{p,q}(\mathbb{R}^{n})}
=\lf\|\lf(\sum_{j\in\nn}\lf|f_j\r|^r\r)^{1/r}\r\|
_{\dot{\mathcal{K}}_{A,\omega}^{p,q}(\mathbb{R}^{n})},
\end{align*}
which completes the proof of Lemma \ref{vmbhl}.
\end{proof}

By Lemma \ref{vmbhl} and an argument similar to
that used in the proof of \cite[Lemma 4.3.10]{lyh22},
we obtain the following conclusion.

\begin{theorem}\label{vmbhl-cor}
Let $A$ be a dilation, $p,q\in(0,\infty)$, and $\omega\in M(\mathbb{R}_{+})$
satisfy $m_{0}(\omega)\in(-\frac{1}{p},\infty)$
and $m_{\fz}(\omega)\in(-\frac{1}{p},\infty)$.
Then, for any given $r\in(1,\infty)$ and
\begin{align}\label{5.25.y1}
u\in\left(0,\min\left\{p,\frac{1}
{\max\{M_{0}(\omega),M_{\fz}(\omega)\}+1/p}\right\}\right),
\end{align}
there exists a positive constant $C$ such that,
for any $\{f_j\}_{j\in\mathbb{N}}\subset L_{{\rm loc}}^{1}(\rn)$,
\begin{equation*}
\left\|\left\{\sum_{j\in\mathbb{N}}
\left[\cm(f_{j})\right]^{r}\right\}^{\frac{1}{r}}
\right\|_{[\dot{\mathcal{K}}_{A,\omega}^{p,q}(\mathbb{R}^{n})]^{1/u}}
\leq C\left\|\left\{\sum_{j\in\mathbb{N}}|f_{j}|^{r}\right\}^{\frac{1}{r}}
\right\|_{[\dot{\mathcal{K}}_{A,\omega}^{p,q}(\mathbb{R}^{n})]^{1/u}}.
\end{equation*}
\end{theorem}

\begin{proof}
On the one hand, from \cite[Lemma 1.3.1]{lyh22}, we deduce that
\begin{align}\label{5.25.x2}
\lf[\dot{\mathcal{K}}_{A,\omega}^{p,q}(\mathbb{R}^{n})\r]^{1/u}
=\dot{\mathcal{K}}_{A,\omega^u}^{p/u,q/u}(\mathbb{R}^{n}).
\end{align}
On the other hand, applying Lemma \ref{5.28.x1},
$m_{0}(\omega)>-\frac{1}{p}$, $m_{\fz}(\omega)>-\frac{1}{p}$,
and \eqref{5.25.y1}, we conclude that
$$\min\{m_0(\omega^u),m_{\fz}(\omega^u)\}
=u\min\{m_0(\omega),m_{\fz}(\omega)\}
>-\frac{1}{(p/u)}$$
and
$$\max\{M_0(\omega^u),M_{\fz}(\omega^u)\}
=u\max\{M_0(\omega),M_{\fz}(\omega)\}
<u\lf(\frac1u-\frac1p\r)=\frac{1}{(p/u)'}.$$
These, combined with \eqref{5.25.x2} and Lemma \ref{vmbhl},
further implies that
\begin{align*}
&\left\|\left\{\sum_{j\in\mathbb{N}}
\left[\cm(f_{j})\right]^{r}\right\}^{\frac{1}{r}}
\right\|_{[\dot{\mathcal{K}}_{A,\omega}^{p,q}(\mathbb{R}^{n})]^{1/u}}\\
&\quad=\left\|\left\{\sum_{j\in\mathbb{N}}
\left[\cm(f_{j})\right]^{r}\right\}^{\frac{1}{r}}
\right\|_{\dot{\mathcal{K}}_{A,\omega^u}^{p/u,q/u}(\mathbb{R}^{n})}
\ls\left\|\left\{\sum_{j\in\mathbb{N}}|f_{j}|^{r}\right\}^{\frac{1}{r}}
\right\|_{\dot{\mathcal{K}}_{A,\omega^u}^{p/u,q/u}(\mathbb{R}^{n})}\\
&\quad=\left\|\left\{\sum_{j\in\mathbb{N}}|f_{j}|^{r}\right\}^{\frac{1}{r}}
\right\|_{[\dot{\mathcal{K}}_{A,\omega}^{p,q}(\mathbb{R}^{n})]^{1/u}},
\end{align*}
which completes the proof of Theorem \ref{vmbhl-cor}.
\end{proof}

The following conclusion shows the boundedness of
certain powered Hardy--Littlewood maximal operator
on the associate space of certain power of
$\dot{\mathcal{K}}_{A,\omega}^{p,q}(\mathbb{R}^{n})$;
see \cite[Lemma 1.8.6]{lyh22} for the standard Euclidean space case.

\begin{theorem}\label{mbhal}
Let $A$ be a dilation, $p,q\in(0,\infty)$,
and $\omega\in M(\mathbb{R}_{+})$ satisfy
$m_{0}(\omega)\in(-\frac{1}{p},\infty)$
and $m_{\fz}(\omega)\in(-\frac{1}{p},\infty)$.
Then, for any given
$$s\in\left(0,\min\left\{p,q,\frac{1}
{\max\{M_{0}(\omega),M_{\fz}(\omega)\}+1/p}\right\}\right)$$
and
$$t\in\left(\max\left\{p,\frac{1}
{\min\{m_{0}(\omega),m_{\fz}(\omega)\}+1/p}\right\},\infty\right],$$
the Herz space $[\dot{\mathcal{K}}_{A,\omega}^{p,q}(\mathbb{R}^{n})]^{1/s}$
is a ball Banach function space and there exists a positive
constant $C$ such that, for any $f\in L_{{\rm loc}}^{1}(\rn)$,
\begin{equation}\label{mbhal1}
\left\|\cm^{((t/s)')}(f)\right\|
_{([\dot{\mathcal{K}}_{A,\omega}^{p,q}(\mathbb{R}^{n})]^{1/s})'}
\leq C \left\|f\right\|
_{([\dot{\mathcal{K}}_{A,\omega}^{p,q}(\mathbb{R}^{n})]^{1/s})'}.
\end{equation}
\end{theorem}

To prove this theorem, we need more preparations.
By Theorem \ref{5.27.x1} and an argument similar to
that used in the proof of \cite[Theorem 1.2.46]{lyh22},
we obtain the following conclusion; we omit the details.

\begin{lemma}\label{5.27.y2}
Let $A$ be a dilation, $p,q\in[1,\infty]$,
and $\omega\in M(\mathbb{R}_{+})$ satisfy
$$-\frac1p<m_0(\omega)\le M_0(\omega)<\frac{1}{p'},$$
where $\frac1p+\frac{1}{p'}=1$.
Then the anisotropic local generalized Herz space
$\dot{\mathcal{K}}_{A,\omega}^{p,q}(\mathbb{R}^{n})$
is a ball Banach function space.
\end{lemma}

\begin{proof}[Proof of Theorem \ref{mbhal}]
Obviously,
$$s<\frac{1}{\max\{M_{0}(\omega),M_{\fz}(\omega)\}+1/p}
\le\frac{1}{M_{0}(\omega)+1/p},$$
which, combined with Lemma \ref{5.28.x1}, further implies that
\begin{align}
M_{0}\lf(\omega^s\r)=sM_{0}\lf(\omega\r)
<s\lf(\frac1s-\frac1p\r)=\frac{1}{(p/s)'}
\end{align}
and
\begin{align}
m_{0}\lf(\omega^s\r)=sm_{0}\lf(\omega\r)
>-\frac{s}{p}=-\frac{1}{p/s}.
\end{align}
From these, an argument similar to that used in the proof of
\cite[Lemma 1.3.1]{lyh22}, the assumptions that $p/s>1$ and
$q/s>1$, and Lemma \ref{5.27.y2}, we infer that
\begin{align}\label{5.27.z1}
\lf[\dot{\mathcal{K}}_{A,\omega}^{p,q}(\mathbb{R}^{n})\r]^{1/s}
=\dot{\mathcal{K}}_{A,\omega^s}^{p/s,q/s}(\mathbb{R}^{n})
\end{align}
is a ball Banach function space.

Next, we prove \eqref{mbhal1}. On the one hand,
by \eqref{5.27.z1} and arguments similar to those used
in the proofs of \cite[Lemma 1.3.1 and Theorem 1.7.9]{lyh22},
we conclude that
\begin{align}\label{5.27.z2}
\lf[\lf(\lf[\dot{\mathcal{K}}_{A,\omega}^{p,q}
(\mathbb{R}^{n})\r]^{1/s}\r)'\r]^{1/(t/s)'}
&=\lf[\dot{\mathcal{K}}_{A,1/\omega^s}^{(p/s)',(q/s)'}
(\mathbb{R}^{n})\r]^{1/(t/s)'}\\
&=\dot{\mathcal{K}}_{A,\omega^{-s(t/s)'}}
^{(p/s)'/(t/s)',(q/s)'/(t/s)'}(\mathbb{R}^{n}).\noz
\end{align}
On the other hand, from Lemma \ref{5.28.x1}
and the assumptions that $$s<\frac{1}{\max\{M_{0}(\omega),M_{\fz}(\omega)\}+1/p}$$
and
$$t>\frac{1}{\min\{m_{0}(\omega),m_{\fz}(\omega)\}+1/p},$$
we deduce that
\begin{align*}
\min\lf\{m_{0}\lf(\omega^{-s(t/s)'}\r),m_{\fz}\lf(\omega^{-s(t/s)'}\r)\r\}
&=-s(t/s)'\max\{M_{0}\lf(\omega\r),M_{\fz}\lf(\omega\r)\}\\
&>-s(t/s)'\lf(\frac1s-\frac1p\r)=-\frac{1}{(p/s)'/(t/s)'}
\end{align*}
and
\begin{align*}
\max\lf\{M_{0}\lf(\omega^{-s(t/s)'}\r),M_{\fz}\lf(\omega^{-s(t/s)'}\r)\r\}
&=-s(t/s)'\min\{m_{0}\lf(\omega\r),m_{\fz}\lf(\omega\r)\}\\
&<-s(t/s)'\lf(\frac1t-\frac1p\r)=1-\frac{1}{(p/s)'/(t/s)'}\\
&=\frac{1}{[(p/s)'/(t/s)']'}.
\end{align*}
These, together with \eqref{5.27.z2} and Corollary \ref{5.23.y2},
further imply that
\begin{align*}
\left\|\cm^{((t/s)')}(f)\right\|
_{([\dot{\mathcal{K}}_{A,\omega}^{p,q}(\mathbb{R}^{n})]^{1/s})'}
&=\left\|\cm\lf(|f|^{(t/s)'}\r)\right\|
_{\dot{\mathcal{K}}_{A,\omega^{-s(t/s)'}}
^{(p/s)'/(t/s)',(q/s)'/(t/s)'}(\mathbb{R}^{n})}\\
&\ls\left\||f|^{(t/s)'}\right\|
_{\dot{\mathcal{K}}_{A,\omega^{-s(t/s)'}}
^{(p/s)'/(t/s)',(q/s)'/(t/s)'}(\mathbb{R}^{n})}\\
&=\left\|f\right\|
_{([\dot{\mathcal{K}}_{A,\omega}^{p,q}(\mathbb{R}^{n})]^{1/s})'}.
\end{align*}
This finishes the proof of \eqref{mbhal1}
and hence Theorem \ref{mbhal}.
\end{proof}

\begin{remark}\label{6.3rem}
Let $p,q\in(0,\infty)$ and $\omega\in M(\mathbb{R}_{+})$
satisfy $m_{0}(\omega)\in(-\frac{1}{p},\infty)$
and $m_{\fz}(\omega)\in(-\frac{1}{p},\infty)$.
Choose
$$p_-:=\min\left\{p,q,
\frac{1}{\max\left\{M_{0}(\omega),M_{\infty}(\omega)\right\}+1/p}\right\},
\ \tz_0\in\left(0,\underline{p}\right),$$
and
$$p_0\in\left(\max\left\{1,p,\frac{1}{\min\left\{m_{0}(\omega),
m_{\infty}(\omega)\right\}+1/p}\right\},\infty\right),$$
where $\underline{p}$ is the same as in \eqref{underp}.
From these and Theorems \ref{5.27.x1}, \ref{vmbhl-cor}, and \ref{mbhal},
we deduce that $\dot{\mathcal{K}}_{A,\omega}^{p,q}(\mathbb{R}^{n})$
satisfies both Assumptions \ref{F-SVMI} and \ref{HMLonAsso}.
Further assume that
$$N\in\nn\cap\lf[\lf\lfloor\lf(\frac{1}{\min\{p,q,
\frac{1}{\max\{M_{0}(\omega),M_{\infty}(\omega)\}+1/p}\}}-1\r)
\frac{\ln b}{\ln(\lz_-)}\r\rfloor+2,\infty\r).$$
Then $\dot{\mathcal{K}}_{A,\omega}^{p,q}(\mathbb{R}^{n})$ satisfies all
the assumptions of Definition \ref{HXA}.
Let $H\dot{\mathcal{K}}_{A,\omega}^{p,q}(\mathbb{R}^{n})$ be the
anisotropic local generalized Herz--Hardy space defined in
Definition \ref{HXA} with $X:=\dot{\mathcal{K}}_{A,\omega}^{p,q}(\mathbb{R}^{n})$.
Then the main results of \cite{wyy22} and \cite{lyy23},
for example, the maximal function characterizations of $H\dot{\mathcal{K}}_{A,\omega}^{p,q}(\mathbb{R}^{n})$
(see \cite[Theorem 3.9]{wyy22}),
the (finite) atomic characterizations of $H\dot{\mathcal{K}}_{A,\omega}^{p,q}(\mathbb{R}^{n})$
(see \cite[Theorems 4.3 and 5.4]{wyy22}),
the molecular characterization of
$H\dot{\mathcal{K}}_{A,\omega}^{p,q}(\mathbb{R}^{n})$
(see \cite[Theorem 6.3]{wyy22}),
the boundedness of Calder\'on--Zygmund operators on $H\dot{\mathcal{K}}_{A,\omega}^{p,q}(\mathbb{R}^{n})$
(see \cite[Theorems 7.10 and 7.11]{wyy22}),
the Littlewood--Paley function characterizations of $H\dot{\mathcal{K}}_{A,\omega}^{p,q}(\mathbb{R}^{n})$
(see \cite[Theorems 5.4, 5.5, and 5.6]{lyy23}),
and the dual theorem of
$H\dot{\mathcal{K}}_{A,\omega}^{p,q}(\mathbb{R}^{n})$
(see \cite[Theorem 3.25]{lyy23}) still hold true.
\end{remark}

Let $p,q\in(0,1)$ and $\omega\in M(\mathbb{R}_{+})$
satisfy $m_{0}(\omega)\in(-\frac{1}{p},\infty)$
and $m_{\fz}(\omega)\in(-\frac{1}{p},\infty)$.
Choose a $q_0=1$. By the above discussion, an argument similar to that
used in the proof of \cite[Lemma 4.9.4]{lyh22}, and \cite[(4.135)]{lyh22},
we find that all the assumptions of Theorems
\ref{s2t1}, \ref{s3t1}, and \ref{s3t2} are satisfied with $X:=\dot{\mathcal{K}}_{A,\omega}^{p,q}(\mathbb{R}^{n})$.
Applying Theorems \ref{s2t1}, \ref{s3t1},
and \ref{s3t2}, we obtain the following result.

\begin{theorem}\label{tholocal}
Let $p,q\in(0,1)$ and $\omega\in M(\mathbb{R}_{+})$
satisfy $m_{0}(\omega)\in(-\frac{1}{p},\infty)$
and $m_{\fz}(\omega)\in(-\frac{1}{p},\infty)$.
Then Theorems \ref{s2t1}, \ref{s3t1}, and \ref{s3t2} still hold true
with $X$ replaced by $\dot{\mathcal{K}}_{A,\omega}^{p,q}(\mathbb{R}^{n})$.
\end{theorem}

\begin{remark}\label{relocal}
We point out that Theorem \ref{tholocal} is completely new.
When $A:=2I_{n\times n}$, the corresponding results of
Theorem \ref{tholocal} can be found in \cite[Theorem 4.9.1]{lyh22}.
\end{remark}


\bigskip

\noindent Chaoan Li and Dachun Yang (Corresponding author)

\medskip

\noindent Laboratory of Mathematics and Complex Systems
(Ministry of Education of China),
School of Mathematical Sciences, Beijing Normal University,
Beijing 100875, The People's Republic of China

\smallskip

\noindent {\it E-mails}:
\texttt{cali@mail.bnu.edu.cn} (C. Li)

\noindent\phantom{{\it E-mails:}}
\texttt{dcyang@bnu.edu.cn} (D. Yang)

\bigskip

\noindent Xianjie Yan

\medskip

\noindent Institute of Contemporary Mathematics,
School of Mathematics and Statistics, Henan University,
Kaifeng 475004, The People's Republic of China

\smallskip

\noindent{{\it E-mail:}}
\texttt{xianjieyan@henu.edu.cn}


\end{document}